\documentclass[11pt]{amsart}

\usepackage{amsmath, amsthm, amssymb, amsfonts}

\usepackage{array}

\usepackage{tikz}
\usetikzlibrary{matrix}

\newcommand{\bE}{\mathbf{E}}
\newcommand{\bI}{\mathbf{I}}

\newcommand{\uk}{\underline{k}}
\newcommand{\ul}{\underline{l}}

\newcommand{\us}{\underline{s}}

\newcommand{\ux}{\underline{x}}

\newcommand{\ZZ}{\mathbb{Z}}
\newcommand{\RR}{\mathbb{R}}
\newcommand{\CC}{\mathbb{C}}
\newcommand{\PP}{\mathbb{P}}

\newcommand{\TT}{\mathbb{T}}

\newcommand{\cC}{\mathcal{C}}
\newcommand{\cO}{\mathcal{O}}
\newcommand{\cE}{\mathcal{E}}

\newcommand{\cG}{\mathcal{G}}
\newcommand{\cH}{\mathcal{H}}
\newcommand{\cI}{\mathcal{I}}
\newcommand{\cS}{\mathcal{S}}

\newcommand{\cR}{\mathcal{R}}

\DeclareMathOperator{\Cl}{Cl}
\DeclareMathOperator{\Div}{div}

\DeclareMathOperator{\cone}{cone}
\DeclareMathOperator{\Hom}{Hom}

\DeclareMathOperator{\ini}{in_{>}}
\DeclareMathOperator{\Spec}{Spec}

\DeclareMathOperator{\sat}{sat}

\DeclareMathOperator{\Pic}{Pic}

\theoremstyle{definition}
\newtheorem{definition}{Definition}[section]
\newtheorem{example}[definition]{Example}
\newtheorem{remark}[definition]{Remark}

\theoremstyle{plain}
\newtheorem{proposition}[definition]{Proposition}
\newtheorem{corollary}[definition]{Corollary}

\newtheorem{lemma}[definition]{Lemma}

\begin{document}

\title{Klyachko diagrams of monomial ideals}

\author[Rosa M. Mir\'o-Roig]{Rosa M. Mir\'o-Roig}
\address{Department de matem\`{a}tiques i Inform\`{a}tica, Universitat de Barcelona, Gran Via de les Corts Catalanes 585, 08007 Barcelona,
Spain}
\email{miro@ub.edu}

\author[Marti Salat-Molt\'o]{Mart\'i Salat-Molt\'o}
\address{Department de matem\`{a}tiques i Inform\`{a}tica, Barcelona Graduate School of Mathematics (BGSMath), Universitat de Barcelona, Gran Via de les Corts Catalanes 585, 08007 Barcelona,
Spain}
\email{marti.salat@ub.edu}

\begin{abstract}
In this paper, we introduce the notion of a {\em Klyachko diagram} for a monomial ideal $I$ in a certain multi-graded polynomial ring,
namely the  Cox ring $R$ of a smooth complete toric variety, with irrelevant maximal ideal $B$. We present procedures to compute the Klyachko diagram of $I$ from its monomial generators, and to retrieve the $B-$saturation $I^{\mathrm{sat}}$ of $I$ from its Klyachko diagram. We use this description to compute the first local cohomology module $H^{1}_{B}(I)$. As an application, we find a formula for the Hilbert function of $I^{\mathrm{sat}}$, and a characterization of monomial ideals with constant Hilbert polynomial, in terms of their Klyachko diagram.
\end{abstract}

\thanks{\noindent Corresponding author: Mart\'i Salat-Molt\'o.\\
Acknowledgements: The first author was partially supported by PID2019-104844GB-I00. The second author is partially supported by  MDM-2014-0445-18-2.}


\maketitle

\tableofcontents

\markboth{}{}


\section{Introduction}


Lying in the crossroads of commutative algebra and combinatorics, monomial ideals play a prominent role in the study of ideals in a polynomial ring $R$. Indeed, many properties of arbitrary ideals $I \subset R$ are reduced to the monomial case, which can often be tackled using combinatorial tools. For instance, it is a classical result due to Macaulay in \cite{Macaulay}, that the Hilbert function of an ideal $I \subset R$ coincides with the Hilbert function of its {\em initial ideal} $\ini(I)$, which is itself a monomial ideal (see for instance \cite[Theorem~15.3]{EisenbudBook}). Since the advent of combinatorial commutative algebra, the theory of monomial ideals has been linked with various topics in discrete mathematics, such as enumerative combinatorics, graph theory, simplicial geometry or lattice polytopes (see \cite{Eliahou-Kervaire, Hei-Rat-Shah, Bayer-Sturmfels, Taylor, Herzog-Hibi, Faridi, Eagon-Reiner}).

The aim of this paper is to introduce the {\em Klyachko diagram} of a monomial ideal, which can be seen as a generalization of the classical {\em staircase diagram}, suited to study monomial ideals inside non-standard graded polynomial rings. More precisely, we focus on the polynomial ring
$R=\CC[x_1, \ldots, x_r]$ graded by the class group $\Cl(X) \cong \ZZ^{\ell}$
of a smooth complete toric variety $X$. That is, $r = |\Sigma(1)|$
is the number of rays of the fan $\Sigma$ of $X$, and the degree of a variable
$x_i$ is the class in $\Cl(X)$ of the torus-invariant Weil divisor $D_{\rho_{i}}$ corresponding to the ray $\rho_i$.
The graded ring $R$ can be considered as the Cox ring of the toric variety $X$, and it appears in the construction of a toric variety by a GIT quotient \cite[Chapter 5]{CLS}.
For instance, if we consider $X=\PP^{r-1}$, we recover the polynomial ring with its classical $\ZZ$-grading.
 
Apart from the Cox ring being a generalization of the classical $\ZZ-$graded polynomial ring, $\Cl(X)-$gra\-ded $R-$modules correspond to quasi-coherent sheaves on $X$. In particular, a $\Cl(X)-$gra\-ded ideal $I\subset R$ corresponds to an ideal sheaf $\cI$ on $X$ such that 
\[
H^{0}_{\ast}(X,\cI)=\bigoplus_{\alpha\in\Cl(X)}H^{0}(X,\cI(\alpha))\cong (I:B^{\infty})=I^{\sat},
\]
where $B$ is the {\em irrelevant ideal} of $X$. It is a monomial maximal ideal determined combinatorially by the fan $\Sigma$ of the toric variety $X$. A $\Cl(X)-$graded $R-$module $E$ gives rise to an equivariant sheaf if and only if $E$ is $\ZZ^{r}-$graded (also called {\em fine-graded}). In particular, equivariant ideal sheaves on $X$ are in correspondence to monomial ideals in $R$.

In \cite{Kly89} and \cite{Kly91}, Klyachko classified equivariant torsion-free sheaves on $X$ in terms of filtered collections of vector spaces. These filtered collections, parameterized by the cones of the fan $\Sigma$, are often referred to in the literature as {\em Klyachko filtrations} (see Proposition \ref{Proposition:Torsionfrees}). In \cite{Perling}, this device was formalized by Perling, who introduced the notion of a $\Sigma-$family, obtaining a general classification of equivariant quasi-coherent sheaves. From a geometrical point of view, these methods have been used in the last two decades to study equivariant vector bundles on toric varieties (see \cite{Her-Mus-Pay, Pay, Das-Dey-Khan, Di-Rocco-Jabbusch-Smith}). On the other hand, in \cite{MR-S}, the present authors used the theory of $\Sigma-$families to study reflexive $\Cl(X)-$gra\-ded $R-$modules from a commutative algebra perspective. 

In this note, we use this construction to introduce the Klyachko diagram of a monomial ideal $I\subset R$: a family of staircase-like diagrams parametrized by the cones of $\Sigma$ encoding algebraic properties of $I$ (see for instance Example \ref{Example:Klyachko diagram algorithm} and Figure \ref{Fig:PP2alt}). In particular, the Klyachko diagram is uniquely determined by the ideal $I$, up to $B-$saturation. We give procedures to compute the Klyachko diagram using the monomial generators of $I$ as initial data and conversely, to determine the generators of a $B-$saturated ideal $I^{\sat}$ from a given Klyachko diagram $\{(\cC^{\sigma}_{I},\Delta^{\sigma}_{I})\}_{\sigma\in\Sigma}$. We also provide a method to compute the first local cohomology module $H^{1}_{B}(I)$ with respect to $B$ from the diagram $\{(\cC^{\sigma}_{I},\Delta^{\sigma}_{I})\}_{\sigma\in\Sigma}$, which measures the saturatedness of $I$. We then use the Klyachko diagram to give a formula for the $\Cl(X)-$graded Hilbert function of $I^{\sat}$ in terms of lattice polytopes. Finally, we characterize monomial ideals $I$ with constant Hilbert polynomial in terms of their Klyachko diagram.

\vspace{3mm}
Next we explain how this paper is organized. Section \ref{Section:Preliminaries} contains all the preliminary results and definitions needed for the rest of this work, and it is divided in two parts. In Subsection \ref{Section:Toric varieties}, we recall the notation and basic results concerning toric varieties. In Subsection \ref{Section:Preliminaries Klyachko filtrations for modules}, we recall the theory of Klyachko filtrations.

The remaining two sections are the main body of the paper. In Section \ref{Section:Klyachko diagrams}, we define the Klyachko diagram of a monomial ideal, and we establish its main properties. In Subsection \ref{Section:From ideal to Klyachko diagram}, we present a procedure to obtain the Klyachko diagram from the generators of a given monomial ideal $I$, and we prove that it describes the collection of Klyachko filtrations of $I$ (Proposition \ref{Prop:From ideal to diagram}). As a corollary, we show how the Klyachko diagram of the sum of two monomial ideals can be computed. Conversely, in Subsection \ref{Section:From Klyachko diagram to ideal}, we give a method to obtain a minimal set of generators of a $B-$saturated monomial ideal corresponding to a given Klyachko diagram. Finally in Subsection \ref{Section:HH1}, we use our previous results to compute the first local cohomology module $H^{1}_{B}(I)$ (Proposition \ref{Corollary:local HH1}) which measures how different $I$ and $I^{\sat}$ are. In the last part of this note, we give a formula for the Hilbert function of a $B-$saturated monomial ideal in terms of its Klyachko diagram (Proposition \ref{Prop:Hilbert function and Klyachko diagram}), and we finish characterizing the Klyachko diagram of monomial ideals with a constant Hilbert polynomial (Corollary \ref{Corollary:HilbPoly}). In particular, we characterize all one dimensional monomial ideals $I\subset R$ in terms of the Klyachko diagram.


\section{Preliminaries}\label{Section:Preliminaries}
In this section, we gather the basic notations, definitions and results about toric varieties needed in the sequel. We recall the notion of a $\Sigma-$family of an equivariant torsion-free sheaf, as introduced in \cite{Perling}, and we end specializing it to the setting of equivariant ideal sheaves.
\subsection{Toric varieties}\label{Section:Toric varieties}

Let $X$ be an $n-$dimen\-sional smooth complete toric variety with torus $\TT_{N}\cong(\CC^{\ast})^{n}$, associated to a fan $\Sigma\subset N\otimes\RR\cong\RR^{n}$, where $N\cong\ZZ^{n}$ is the cocharacter lattice of $\TT_{N}$. We denote by $\Sigma(k)$ (respectively $\sigma(k)$) the set of $k-$dimensional cones in $\Sigma$ (respectively in $\sigma$). We refer to the cones $\rho\in\Sigma(1)$ as {\em rays} and we set $n(\rho)\in N$ to be the first non-zero lattice point along $\rho$. We denote by $M=\Hom(N,\ZZ)\cong\ZZ^{n}$ its {\em character lattice} and for $m\in M$, we set $\chi^{m}:\TT_{N}\rightarrow\CC^{\ast}$ the corresponding algebraic group homomorphism. For any cone $\sigma\in\Sigma$, let $\sigma^{\vee}$ be its {\em dual cone}, let $S_{\sigma}:=\sigma^{\vee}\cap M$ be the associated semigroup of characters and $\CC[S_{\sigma}]$ the corresponding $\CC-$algebra. Then $U_{\sigma}=\Spec(\CC[S_{\sigma}])\subset X$ is a $\TT_{N}-$invariant affine subvariety of $X$. For any two cones $\tau\prec \sigma\in \Sigma$, there is a character $m\in M$ such that $S_{\tau}=S_{\sigma}+\ZZ\langle m\rangle$ and we have an inclusion $U_{\tau}\hookrightarrow U_{\sigma}$ given by the natural morphism of $\CC-$algebras $\CC[S_{\sigma}]\hookrightarrow \CC[S_{\sigma}]_{\chi^{m}}=\CC[S_{\tau}]$.

There is a bijection between rays $\rho \in \Sigma(1)$ and $\TT_{N}-$invariant Weil divisors $D_{\rho}$. Furthermore, the $\TT_{N}-$invariant Weil divisors generate the class group $\Cl(X)$ of $X$. Indeed, we have the exact sequence
\begin{equation}\label{Eq:Class group exact sequence}
0\rightarrow M\xrightarrow{\phi}\bigoplus_{\rho\in\Sigma(1)}\ZZ D_{\rho}\xrightarrow{\pi} \Cl(X)\rightarrow 0,
\end{equation}
where $\phi(m)=\Div(\chi^{m})=\sum_{\rho\in\Sigma(1)}\langle m, n(\rho)\rangle D_{\rho}$, for any character $m\in M$; and $\pi(D)=[D]\in \Cl(X)$ the class of an invariant Weil divisor $D$. Hence, $\Cl(X)$ is a finitely generated abelian group. (See \cite[Theorem 4.1.3]{CLS}).
 
Let $R=\CC[x_{\rho}\;|\;\rho\in\Sigma(1)]$ be a polynomial ring in $|\Sigma(1)|$ variables. The {\em Cox ring} of $X$ is the $\CC-$algebra $R$ endowed with a grading, not necessarily standard,  given by the class group $\Cl(X)$ of $X$. We set $\deg(x_{\rho}):=[D_{\rho}]\in\Cl(X)$, for each ray $\rho\in\Sigma(1)$. We write $R=\CC[x_{1},\dotsc,x_{r}]$ whenever $\Sigma(1)=\{\rho_{1},\dotsc,\rho_{r}\}$ is the (ordered) set of rays of $\Sigma$. For a cone $\sigma$, we set 
\[
x^{\hat{\sigma}}:=\prod_{\rho_{i}\in \Sigma(1)\setminus \sigma(1)}x_{i},\quad\text{and}\quad B:=\langle x^{\hat{\sigma}}\;\mid\;\sigma\in\Sigma\rangle.
\]
$B$ is called the {\em irrelevant ideal}. In fact, one has $B=\langle x^{\hat{\sigma}}\;\mid\;\sigma\in\Sigma_{\max}\rangle$. 

\begin{remark}
In general, the Cox ring can be defined for any variety $X$ as the ring
\[
\cR(X)=\bigoplus_{[D]\in\Pic(X)}H^{0}(X,\cO(D)).
\]
In the special case when $X$ is a smooth toric variety it coincides with the polynomial ring we defined above.
\end{remark}

\begin{example} (i) \label{Example:Projective space and Hirzebruch}
$\PP^{n}$ is a toric variety of dimension $n$. Let $\{e_{1},\dotsc,e_{n}\}$ be a basis of $N=\ZZ^{n}$. The fan $\Sigma$ associated to $\PP^{n}$ has $n+1$ rays: $\rho_{0}=\cone(-e_{1}-\dotsb-e_{n})$ and $\rho_{i}=\cone(e_{i})$ for $1\leq i\leq n$; and $n+1$ maximal cones $\sigma_{0}:=\cone(e_{1},\dotsc,e_{n})$ and $\sigma(i):=\cone(\{e_{j}\;\mid\;j\neq i\}\cup\{-e_{1}-\dotsb-e_{n}\})$ for $1\leq i\leq n$. Its  associated Cox ring is $\CC[x_{0},\dotsc,x_{n}]$ with $\deg(x_{i})=1$ for $0\leq i\leq n$, and its irrelevant ideal is $B =\langle x_{0},\dotsc,x_{n}\rangle$.

(ii) For $a\geq0$, the Hirzebruch surface $\cH_{a}=\PP(\cO_{\PP^{1}}\oplus\cO_{\PP^{1}}(a))$ is a toric surface.
Let $N=\ZZ^{2}$ be a lattice with $\{e,f\}$ its standard basis, and set $u_{0}:=-e+af$, $u_{1}:=e$, $v_{0}:=-f$ and $v_{1}:=f$. The fan $\Sigma$ associated to $\cH_{a}$ has four rays $\rho_{0}=\cone(u_{0})$, $\rho_{1}=\cone(u_{1})$, $\eta_{0}=\cone(v_{0})$ and $\eta_{1}=\cone(v_{1})$; and four maximal cones $\sigma_{00}=\cone(u_{1},v_{1})$, $\sigma_{01}=\cone(u_{1},v_{0})$, $\sigma_{10}=\cone(u_{0},v_{1})$ and $\sigma_{11}=\cone(u_{0},v_{0})$. Its Cox ring is $\CC[x_{0},x_{1},y_{0},y_{1}]$ with $\deg(x_{0})=\deg(x_{1})=(1,0)$, $\deg(y_{0})=(0,1)$ and $\deg(y_{1})=(-a,1)$; and its irrelevant ideal is $B(\Sigma)=\langle x_{1}y_{1}, x_{1}y_{0}, x_{0}y_{1}, x_{0}y_{0} \rangle$.
\end{example}

For any cone $\sigma\in \Sigma$, the localization of $R$ at $x^{\hat{\sigma}}$ is a $\Cl(X)-$graded algebra $R_{x^{\hat{\sigma}}}$. For any Weil divisor $D = \sum_{\rho \in \Sigma(1)}a_{\rho}D_{\rho}$, there is an isomorphism between $\CC[S_{\sigma}]$ and the homogeneous $[D]-$graded piece $(R_{x^{\hat{\sigma}}})_{[D]}$ sending $\chi^{m}\in\CC[S_{\sigma}]$ to the monomial $\ux^{m+D}:=\prod_{\rho \in \Sigma(1)} x_{\rho}^{\langle m, \rho\rangle+a_{\rho}}\in(R_{x^{\hat{\sigma}}})_{[D]}$. We have the following: 

\begin{proposition} \label{Proposition:line bundles and quasi-coherent sheaves} \begin{itemize}
\item[(i)] For any $\alpha \in \Cl(X)$, there is a natural isomorphism $R_{\alpha} \cong \Gamma(X,\mathcal{O}_{X}(D))$ for any Weil divisor $D = \sum_{\rho} a_{\rho}D_{\rho}$ such that $\alpha = [D]$.

\item[(ii)] If $E$ is a $\Cl(X)-$graded $R$-module, there is a quasi-coherent sheaf $\widetilde{E}$ on $X$ such that $\Gamma(U_{\sigma},\widetilde{E}) = (E_{x^{\hat{\sigma}}})_{0}$, for any $\sigma \in \Sigma$.

\item[(iii)] If $\cE$ is a quasicoherent sheaf on $X$, there is a $\Cl(X)-$graded $R$--module such that $\widetilde{E}=\cE$. $\widetilde{E}$ is coherent if and only if $E$ is finitely generated. 

\item[(iv)] $\widetilde{E}=0$ if and only if $B^{l}E=0$ for all $l\gg 0$.

\item[(v)] There is an exact sequence of $\Cl(X)-$graded modules
\[
0\rightarrow H^{0}_{B}(E)\rightarrow E \rightarrow H^{0}_{\ast}(X,\widetilde{E})\rightarrow H^{1}_{B}(E)\rightarrow0.
\]
\end{itemize}

\end{proposition}
\begin{proof}
(i)--(iv) follow from  \cite[Proposition 5.3.3, Proposition 5.3.6, Proposition 5.3.7 and Proposition 5.3.10]{CLS}. (v) follows from \cite[Proposition 2.3]{Eis-Mus-Sti} 
\end{proof}

The module $\Gamma E = H^{0}_{\ast}(X,\widetilde{E})$ is called the {\em $B-$saturation} of $E$. We say that $E$ is {\em $B-$saturated} if $E\cong \Gamma E$, or equivalently if $H^{0}_{B}(E)=H^{1}_{B}(E)=0$. If $H^{0}_{B}(E)=(0:_{E}B^{\infty})=0$, we say that $E$ is $B-$torsion free.


\subsection{Equivariant sheaves and Klyachko filtrations}
\label{Section:Preliminaries Klyachko filtrations for modules}
Let $X$ be a smooth complete toric variety with fan $\Sigma$ and $R=\CC[x_{1},\dotsc,x_{r}]$ its associated $\Cl(X)-$graded Cox ring. In this subsection, we introduce the notion of a $\Sigma-$family to describe equivariant sheaves on $X$. We refer the reader to \cite{Perling} and \cite{Kly89} for further details. 
\begin{definition}
For any $t\in\TT_{N}$, let $\mu_{t}:X\rightarrow X$ be the morphism given by the action of $\TT_{N}$ on $X$. A quasi-coherent sheaf $\cE$ on $X$ is {\em equivariant} if there is a family of isomorphisms $\{\phi_{t}:\mu_{t}^{\ast}\cE\cong\cE\}_{t\in\TT_{N}}$ such that $\phi_{t_{1}\cdot t_{2}}=\phi_{t_{2}}\circ\mu_{t_{2}}^{\ast}\phi_{t_{1}}$ for any $t_{1},t_{2}\in\TT_{N}$.
\end{definition}
Notice that any $\ZZ^{r}-$graded $R-$module is also $\Cl(X)-$graded. In \cite{Batyrev-Cox}, Batyrev and Cox proved the following result:
\begin{proposition}\label{Proposition:Multigraded-Equivariant}
Let $E$ be a $\Cl(X)-$graded $R-$module. The quasi-coherent sheaf $\widetilde{E}$ is equivariant if and only if $E$ is also $\ZZ^{r}-$graded.
\end{proposition}
\begin{proof}
See \cite[Proposition 4.17]{Batyrev-Cox}.
\end{proof}
In \cite{Kly89} and \cite{Kly91}, Klyachko observed that to any equivariant torsion-free sheaf we can associate a family of filtered vector spaces, the so-called Klyachko filtration. In what follows we recall how this family can be constructed. Let $\cE$ be an equivariant sheaf on $X$ corresponding to a $\ZZ^{r}-$graded module $E$.
For any degree $\alpha\in\Cl(X)$, the exact sequence (\ref{Eq:Class group exact sequence}) endows the homogeneous degree$-\alpha$ piece of $E$ with an $M-$grading:
\[
E_{\alpha}=\bigoplus_{m\in M}E_{z+\phi(m)},\quad\text{for any}\quad z\in\pi^{-1}(\alpha).
\]
Now, for any $\sigma\in \Sigma$ we consider the monomial $x^{\hat{\sigma}}$, and the localized $R_{x^{\sigma}}-$mo\-dule $E_{x^{\sigma}}$ remains $\ZZ^{r}-$graded. As before, for any $\alpha\in\Cl(X)$, $(E_{x^{\hat{\sigma}}})_{\alpha}$ is $M-$graded. In particular, taking $\alpha=0$ we have:
\begin{equation}\label{Eq:Sigma family decomposition}
E^{\sigma}:=(E_{x^{\hat{\sigma}}})_{0}=\bigoplus_{m\in M}(E_{x^{\hat{\sigma}}})_{\phi(m)}=:\bigoplus_{m\in M}E^{\sigma}_{m}.
\end{equation}
Since $(E_{x^{\hat{\sigma}}})_{0}$ is isomorphic to the $\CC[S_{\sigma}]-$module $\Gamma(U_{\sigma},\cE)$, geometrically we can see (\ref{Eq:Sigma family decomposition}) as the isotypical decomposition of $\Gamma(U_{\sigma},\cE)$ into $\TT_{N}-$eigenspaces of sections
\[
\Gamma(U_{\sigma},\cE)=\bigoplus_{m\in M}\Gamma(U_{\sigma},\cE)_{m}.
\]

Recall that the semigroup $S_{\sigma}$ induces a preorder on the character lattice $M$: for any $m,m'\in M$ we say that $m\leq_{\sigma} m'$ if and only if $m'-m\in S_{\sigma}$, or equivalently if $\langle m' - m,u\rangle\geq0$ for all $u\in\sigma$.
For any two characters $m\leq_{\sigma}m'$, the multiplication by $\chi^{m'-m}\in \CC[S_{\sigma}]$ yields the map
\[\chi^{\sigma}_{m,m'}:E^{\sigma}_{m}\rightarrow E^{\sigma}_{m'}.
\]
For any $m\leq_{\sigma}m'\leq_{\sigma}m''$, we have 
\[
\chi^{\sigma}_{m,m}=1\quad\text{and}\quad\chi^{\sigma}_{m,m''}=\chi^{\sigma}_{m',m''}\circ\chi^{\sigma}_{m,m'}.
\]
In particular,
$\chi^{\sigma}_{m,m'}$ is an isomorphism if $m\leq_{\sigma} m'$ and $m'\leq_{\sigma}m$, or equivalently if $m'-m\in\sigma^{\perp}$.
We call $\hat{E}^{\sigma}:=\{E^{\sigma}_{m},\chi_{m,m'}^{\sigma}\}$ a {\em $\sigma-$family} (see \cite[Definition 4.2]{Perling}).

On the other hand, let $\tau\prec\sigma$ be two cones in $\Sigma$ and $m\in M$ the character such that $S_{\tau}=S_{\sigma}+\ZZ\langle m\rangle$. There are isomorphisms $\CC[S_{\tau}]\cong\CC[S_{\sigma}]_{\chi^{m}}$ and $E^{\tau}\cong E^{\sigma}_{\chi^{m}}$ given by the localization at $\chi^{m}$. Thus, we have a morphism $i^{\sigma \tau}:E^{\sigma}\rightarrow E^{\tau}$, corresponding geometrically to the restriction map of section modules $\Gamma(U_{\sigma},\cE)\rightarrow\Gamma(U_{\tau},\cE)$. For any character $m'\in M$, the morphism $i^{\sigma\tau}$ induces a linear map
\[
i^{\sigma \tau}_{m'}:E^{\sigma}_{m'}\rightarrow E^{\tau}_{m'}.
\]
We call $\{\hat{E}^{\sigma}\}_{\sigma\in \Sigma}$ a {\em $\Sigma-$family} (see \cite[Definition 4.8]{Perling}). In \cite[Theorem 4.9]{Perling} it is proved that $\Sigma-$families characterize equivariant sheaves on $X$ or equivalently, $B-$saturated $R-$modules. When $\cE$ is torsion-free, we have the following result.
\begin{proposition}\label{Proposition:Torsionfrees}
Let $\cE$ be an equivariant torsion-free sheaf of rank $s$ and $\{\hat{E}^{\sigma}\}$ its associated $\Sigma-$family. The following holds:
\begin{itemize}
\item[(i)] For any $m'\leq_{\sigma} m$, the linear map $\chi_{m',m}^{\sigma}:E^{\sigma}_{m'}\rightarrow E^{\sigma}_{m}$ is injective.
\item[(ii)] For any character $m\in M$, and any cones $\tau\prec\sigma$ in $\Sigma$, the linear map $i^{\sigma\tau}_{m}:E^{\sigma}_{m}\rightarrow E^{\tau}_{m}$ is injective.
\item[(iii)] There is a vector space $\bE\cong\CC^{s}$ such that $E^{\{0\}}_{m}\cong \bE$ for any $m\in M$. 
\end{itemize}
We have the following commutative diagram:
\vspace{-1mm}
\begin{center}
\begin{tikzpicture}
\matrix (M) [matrix of nodes, row sep=0.2cm, column sep=1.8cm, align=center,text width=0.8cm, text height=0.3cm, anchor=center]{
              & $E^{\{0\}}_{m'}$ & $E^{\sigma}_{m'}$ \\
$\bE$ &                &                 \\
              & $E^{\{0\}}_{m}$  & $E^{\sigma}_{m}.$  \\
};
\draw[->] (M-1-3) -- (M-3-3) node[midway,right]{\scriptsize{$\chi^{\sigma}_{m',m}$}};
\draw[->] (M-1-2) -- (M-3-2) 
	node[midway,right]{\scriptsize{$\chi^{\{0\}}_{m',m}$}} 
	node[midway,left]{\large{$\cong\quad\;$}};
\draw[->] (M-1-3) -- (M-1-2);
\draw[->] (M-3-3) -- (M-3-2);
\draw[->] (M-3-3) -- (M-3-2);
\draw[->] (M-3-3) -- (M-3-2);
\draw[->] (M-1-2) -- (M-2-1)
	node[midway, above]{\scriptsize{$\varphi_{m'}$}};
\draw[->] (M-3-2) -- (M-2-1)
	node[midway, below]{\scriptsize{$\varphi_{m}$}};
\end{tikzpicture}
\end{center}
Moreover, for any character $m\in M$, we have
\[
H^{0}(X,\cE)_{m}=\bigcap_{\sigma\in\Sigma_{\max}}E^{\sigma}_{m}.
\]
\end{proposition}
\begin{proof}
See \cite[Section 4.4]{Perling} and \cite[Section 1.2 and 1.3]{Kly91}.
\end{proof}

\begin{remark}
(i) By Proposition \ref{Proposition:Torsionfrees}, the $\Sigma-$family $\{\hat{E}^{\sigma}\}_{\sigma\in\Sigma}$ of a torsion-free sheaf $\cE$ of rank $s$ can be seen as a filtered collection of linear subspaces of a fixed ambient vector space $\bE$. Geometrically, the vector space $\bE$ can be identified with the $s-$dimensional vector space $\Gamma(\TT_{N},\cE)_{m}$ for any character $m\in M$.

(ii) The description of equivariant torsion-free sheaves given above is based on \cite[Section 4]{Perling}. We note that our order of filtrations is reverse of that of Klyachko \cite{Kly89,Kly91}. In these references, the filtration is taken as a collection of linear subspaces of $\cE(x_{0})$, the fiber of $\cE$ at a point in the open orbit $U_{\{0\}}=\TT_{N}\subset X$ (see \cite[Remark 4.25]{Perling}).
\end{remark}

\begin{definition}\label{Def:KlyFilt}
Let $\cE$ be an equivariant torsion-free sheaf, the filtered collection of vector spaces $\{E^{\sigma}_{m}\!\mid\! m\!\in\! M\}_{\sigma\in\Sigma}$ given by its $\Sigma-$family is called the collection of {\em Klyachko filtrations} of $\cE$.
\end{definition} 

In this note we focus on monomial ideals $I$ in the $\Cl(X)-$graded Cox ring $R$. Since monomial ideals are naturally $\ZZ^{r}-$graded they correspond to torsion-free equivariant sheaves of rank $1$.
Therefore, Proposition \ref{Proposition:Torsionfrees} shows that the $\Sigma-$family of a monomial ideal $I$ is structured as a system of vector space filtrations of a $1-$dimensional vector space $\bI\cong\CC$, which can be identified with $I^{\{0\}}_{m}$ for any character $m\in M$.

\begin{remark}\label{Remark:Klydec of monomial ideals}
Let $\{I^{\sigma}_{m}\!\mid\! m\!\in\! M\}_{\sigma\in\Sigma}$ be the collection of Klyachko filtrations of a monomial ideal $I$. Let $m_{0}\in M$ be a character and identify $\bI$ with $I^{\{0\}}_{m_{0}}$.
For each $\sigma\in\Sigma$ and $m\in M$, the linear subspace $I^{\sigma}_{m}\subset \bI$ can be either $I^{\sigma}_{m}\cong \bI$ or $I^{\sigma}_{m}=0$. Therefore, the collection of Klyachko filtrations of a monomial ideal is characterized by attaching to each cone $\sigma\in\Sigma$, the set of characters $\{m\in M\mid I^{\sigma}_{m}\neq0\}$.
\end{remark}

We finish this preliminary section with an example which illustrates Proposition \ref{Proposition:Torsionfrees}, and shows how to compute the collection of Klyachko filtrations of a monomial ideal.

\begin{example}\label{Example:monomial ideal}
Let $R=\CC[x_{0},x_{1},x_2]$ be the Cox ring of $\PP^{2}$ with fan $\Sigma$ as in Example \ref{Example:Projective space and Hirzebruch}(i). Consider the monomial ideal $I=(x_{2}^2,x_{0}x_{2},x_{0}x_{1})$, we will compute the $\Sigma-$family associated to $I$. We present $I$ as follows:
\begin{equation}\label{Eq:Example monomial ideal}
R(0,0,-2)\oplus R(-1,0,-1)\oplus R(-1,-1,0)\xrightarrow{(x_{2}^2\;x_{0}x_{2}\;x_{0}x_{1})}I\rightarrow0.
\end{equation}
Next, we localize at $x^{\widehat{\{0\}}}=x_{0}x_{1}x_{2}$ and we set $R^{\{0\}}:=R_{x^{\widehat{\{0\}}}}$ the localized ring. For any multidegree $(\alpha_{0},\alpha_{1},\alpha_{2})\in\ZZ^{3}$, $R^{\{0\}}_{(\alpha_{0},\alpha_{1},\alpha_{2})}=\CC\langle x_{0}^{\alpha_{0}}x_{1}^{\alpha_{1}}x_{2}^{\alpha_{2}}\rangle$, the vector space spanned by the monomial $x_{0}^{\alpha_{0}}x_{1}^{\alpha_{1}}x_{2}^{\alpha_{2}}$. On the other hand, any character $m=(d_{1},d_{2})$ is embedded as $m=(-d_{1}-d_{2},d_{1},d_{2})$ in $\ZZ^{3}$ via the exact sequence (\ref{Eq:Class group exact sequence}). To compute $I^{\{0\}}_{m}$ we take the degree $m$ component of (\ref{Eq:Example monomial ideal}). This yields the following exact sequence of vector spaces
\[
R^{\{0\}}(0,0,-2)_{m}\oplus R^{\{0\}}(-1,0,-1)_{m}\oplus R^{\{0\}}(-1,-1,0)_{m}\!\xrightarrow{(x_{2}^2\;x_{0}x_{2}\;x_{0}x_{1})}\!I^{\{0\}}_{m}\!\rightarrow\!0.
\]
Thus, $I^{\{0\}}_{m}=\CC\langle x_{0}^{-d_{1}-d_{2}}x_{1}^{d_{1}}x_{2}^{d_{2}}\rangle$ and there are isomorphisms $\phi^{\{0\}}_{m}:I^{\{0\}}_{m}\cong \bI$. 
Let us now fix the ray $\rho_{0}\in\Sigma(1)$ and compute $I^{\rho_{0}}_{m}$ for any character $m=(d_{1},d_{2})\in\ZZ^{2}$. As before, we set $R^{\rho_{0}}:=R_{x^{\widehat{\rho_{0}}}}$ the localization at $x^{\widehat{\rho_{0}}}=x_{1}x_{2}$. Now, for any multidegree $(\alpha_{0},\alpha_{1},\alpha_{2})\in\ZZ^{3}$, 
\[
R^{\rho_{0}}_{(\alpha_{0},\alpha_{1},\alpha_{2})}=
\left\{
\begin{array}{ll}
\CC\langle x_{0}^{\alpha_{0}}x_{1}^{\alpha_{1}}x_{2}^{\alpha_{2}}\rangle,&\text{if}\quad \alpha_{0}\geq0\\
0,&\text{if}\quad \alpha_{0}\leq-1
\end{array}
\right.
\]
and restricting the exact sequence (\ref{Eq:Example monomial ideal}) to degree $m=(d_{1},d_{2})$, we have
\[
I^{\rho_{0}}_{m}=
\left\{
\begin{array}{ll}
\CC\langle x_{0}^{-d_{1}-d_{2}}x_{1}^{d_{1}}x_{2}^{d_{2}}\rangle\cong \bI,&\text{if}\quad -d_{1}-d_{2}\geq0\\
0,&\text{if}\quad -d_{1}-d_{2}\leq-1.
\end{array}
\right.
\]
Similarly, we obtain
\[
I^{\rho_{1}}_{m}\cong
\left\{
\begin{array}{ll}
\bI,&\text{if}\quad d_{1}\geq0\\
0,&\text{if}\quad d_{1}\leq-1
\end{array}
\right.
\quad
I^{\rho_{2}}_{m}\cong
\left\{
\begin{array}{ll}
\bI,&\text{if}\quad d_{2}\geq0\\
0,&\text{if}\quad d_{2}\leq-1.
\end{array}
\right.
\]
It only remains to compute the components in the $\Sigma-$family associated to the two dimensional cones in $\Sigma$. Let us consider $\sigma_{0}\in\Sigma(2)$ with rays $\sigma_{0}(1)=\{\rho_{1},\rho_{2}\}$. We set $R^{\sigma_{0}}:=R_{x^{\widehat{\sigma_{0}}}}$ the localization at $x^{\widehat{\sigma_{0}}}=x_{0}$ and for any multidegree $(\alpha_{0},\alpha_{1},\alpha_{2})\in\ZZ^{3}$, 
\[
R^{\sigma_{0}}_{(\alpha_{0},\alpha_{1},\alpha_{2})}=
\left\{
\begin{array}{ll}
\CC\langle x_{0}^{\alpha_{0}}x_{1}^{\alpha_{1}}x_{2}^{\alpha_{2}}\rangle,&\text{if}\quad \alpha_{1}\geq0,\;\alpha_{2}\geq0\\
0,&\text{if}\quad \alpha_{1}\leq-1\;\text{or}\;\alpha_{2}\leq -1.
\end{array}
\right.
\]
As before, taking the component of degree $m=(d_{1},d_{2})$ of (\ref{Eq:Example monomial ideal}) we obtain
\[
I^{\sigma_{0}}_{m}\cong
\left\{
\begin{array}{llll}
\bI,&\text{if}& d_{1}=0\;\text{and}\;d_{2}\geq1,\;\text{or}\;\\
&&d_{1}\geq1\;\text{and}\;d_{2}\geq0 \\
0,&& \text{otherwise}.
\end{array}
\right.
\]
Similarly, we obtain the remaining components of the $\Sigma-$family:
\[
I^{\sigma_{1}}_{m}\cong
\left\{
\begin{array}{lll}
\bI,&\text{if}& -d_{1}-d_{2}=0\;\text{and}\;d_{2}\geq2,\;\text{or}\;\\
&&-d_{1}-d_{2}\geq1\;\text{and}\;d_{2}\geq0 \\
0,&&\text{otherwise}
\end{array}
\right.
\]
\[
I^{\sigma_{2}}_{m}\cong
\left\{
\begin{array}{ll}
\bI,&\text{if}\quad -d_{1}-d_{2}\geq0\;\text{and}\;d_{1}\geq0\\
0,& \text{otherwise}.
\end{array}
\right.
\]
\end{example}

\section{Klyachko diagrams of monomial ideals}\label{Section:Klyachko diagrams}

In this section, we focus our attention on monomial ideals $I$ in the Cox ring $R$ of a smooth complete toric variety $X$. Using the theory of Klyachko filtrations, we define the {\em Klyachko diagram} of $I$, and we show how it is determined combinatorially by the monomials generating $I$. Conversely, we give a method to compute a minimal set of generators of a $B-$saturated monomial ideal $I$ from its Klyachko diagram. Finally, we compute the first local cohomology module $H_{B}^1(I)$ for any monomial ideal $I$ using its Klyachko diagram.

From now on, we fix a smooth complete toric variety $X$ with fan $\Sigma$. We set  $r = |\Sigma(1)|$, we denote by $R=\CC[x_{1},\dotsc,x_{r}]$ its associated $\Cl(X)-$graded Cox ring and by $B$ its irrelevant ideal.

\subsection{From a monomial ideal to a Klyachko diagram}\label{Section:From ideal to Klyachko diagram}

Let $I=(m_{1},\dotsc,m_{t})$ be a monomial ideal. We write the monomials 
\[
m_{i}=x_{1}^{k_{1}^{i}}\dotsb x_{r}^{k_{r}^{i}}=:\ux^{\uk^{i}},\quad\text{for}\quad\uk^{i}:=(k_{1}^{i},\dotsc,k_{r}^{i})\in\ZZ^{r}_{\geq0},\quad\text{and}\quad 1\leq i \leq t.
\]
and we present $I$ as the image of a $\ZZ^{r}-$graded map as follows:
\begin{equation}\label{Eq:Presentation monomial ideal}
\bigoplus_{i=1}^{t} R(-\uk^{i})\xrightarrow{(m_{1},\dotsc,m_{t})} I\longrightarrow 0.
\end{equation}
Thus, for any character $m=(d_{1},\dotsc,d_{n})$, $I^{\{0\}}_{m}\cong \CC\langle x_{1}^{\langle m, n(\rho_{1})\rangle}
\dotsb x_{r}^{\langle m, n(\rho_{r})\rangle}\rangle$ and there are isomorphisms $\phi^{\{0\}}_{m}:I^{\{0\}}_{m}\cong\bI$. 
Our first objective is to describe the subspaces $I^{\sigma}_{m} \subset \bI$ for any character $m = (d_1,\hdots,d_n)$ and any cone $\sigma \in \Sigma$. As observed in Remark \ref{Remark:Klydec of monomial ideals}, we want to characterize the sets of characters $\{m\in M\mid I^{\sigma}_{m}\neq0$\} for each cone $\sigma\in\Sigma$. Each of this sets can be seen as the staricase diagram for the inclusion $I^{\sigma}\subset R^{\sigma}$ as $\CC[S_{\sigma}]-$modules.

\begin{lemma}\label{Lemma:KlyFilt for monomial}
Let $I=(\ux^{\uk})\subset R$, with $\uk\in\ZZ^{r}_{\geq0}$, be an ideal generated by a single monomial. Then, for any cone $\sigma=\cone(\rho_{i_{1}},\dotsc,\rho_{i_{c}})\in\Sigma$,
\[
I^{\sigma}_{m}\cong
\left\{
\begin{array}{lll}
\bI,&\text{if} &m\in \cC^{\sigma}_{\uk}\\
0,&&\text{otherwise},
\end{array}
\right.
\]
where $\cC_{\uk}^{\sigma}:=\{m\in M\mid\langle m, \rho_{i_{j}}\rangle \geq k_{i_{j}},\ \text{for}\ 1\leq j\leq c\}$.
\end{lemma}
\begin{proof}
The lemma follows from (\ref{Eq:Presentation monomial ideal}), when $t=1$ and using that
\[
R^{\sigma}_{m}(-\uk)=
\left\{
\begin{array}{@{}l@{\ }l@{\ }l@{}}
\CC\langle x_{1}^{\langle m,\rho_{1}\rangle - k_{1}}\!\!\dotsb x_{r}^{\langle m,\rho_{r}\rangle - k_{r}}\rangle, &\text{if}
&\langle m, \rho_{i_{j}}\rangle - k_{i_{j}}\geq 0,\ \text{for}\ 1\leq j\leq c\\
0, &&\text{otherwise}.
\end{array}
\right.
\]
\end{proof}
We set $\cC_{0}^{\sigma}:=\cC^{\sigma}_{(0,\dotsc,0)}$, and notice the inclusion $\cC^{\sigma}_{\uk}\subset \cC^{\sigma}_{0}$ corresponding to $(\ux^{\uk})^{\sigma}\subset R^{\sigma}$ for any $\uk\in\ZZ^{r}$. Applying Lemma \ref{Lemma:KlyFilt for monomial}, repeatedly, we have:

\begin{proposition}\label{Prop:KlyFilt for monomomial ideal}
Let $I = (m_1,\hdots, m_{t}) \subset R$ be a monomial ideal with $m_{i}=\ux^{\uk^{i}}$ for $\uk^{i}\in\ZZ^{r}_{\geq0}$ and $1\leq i \leq t$. Then, for any cone $\sigma\in\Sigma$,
\[
I^{\sigma}_{m}\cong
\left\{
\begin{array}{@{}lll}
\bI, &\text{if} &m\in\bigcup_{i=1}^{t}\cC^{\sigma}_{\uk^{i}}\\
0,& &\text{otherwise}.
\end{array}
\right.
\]
\end{proposition}
\begin{proof}
It follows from (\ref{Eq:Presentation monomial ideal}) that $I^{\sigma}_{m}=0$ if and only if $R^{\sigma}_{m}(-\uk^{i})= 0$ for all $1\leq i \leq t$. By Lemma \ref{Lemma:KlyFilt for monomial} this occurs if and only if $m\in M\setminus\bigcup_{i=1}^{t}\cC^{\sigma}_{\uk}$, and the result follows.
\end{proof}

Notice that Proposition \ref{Prop:KlyFilt for monomomial ideal} already gives a description of the collection of Klyachko filtrations of a monomial ideal. However, the information on the inclusion $I^{\sigma}\subset R^{\sigma}$ is encoded in 
\[\cC^{0}\setminus\bigcup_{i=1}^{t}\cC^{\sigma}_{\uk^{i}}=\bigcap_{i=1}^{t}(\cC^{\sigma}_{0}\setminus \cC^{\sigma}_{\uk^{i}})=\bigcap_{i=1}^{t}\bigcup_{j=1}^{c}\{m\in M \mid 0\leq \langle m, \rho_{i_{j}} \rangle < k^{i}_{i_{j}}\},\]
which is the union of $c^{t}$ sets. Indeed, for each $1\leq j_{1},\dotsc,j_{t}\leq c$,
\[
P_{j_{1},\dotsc,j_{t}}:=\{m\in M\mid 0\leq \langle m, \rho_{i_{j_{1}}}\rangle < k^{1}_{j_{1}},\dotsc,0\leq \langle m, \rho_{i_{j_{t}}}\rangle < k^{t}_{j_{t}}\}.
\]
The {\em Klyachko diagram} defined below is used in Proposition \ref{Prop:From ideal to diagram} to give a more compact alternative characterization of the collection of Klyachko filtrations of a monomial ideal. We attach to the monomial ideal $I$ a collection of pairs $\{(\cC^{\sigma}_{I},\Delta^{\sigma}_{I})\}_{\sigma\in\Sigma}$ constructed as follows.

For any ray $\rho_{j}\in\Sigma(1)$, we set $s_{j}=s_{\rho_{j}}:=\min\{\deg_{\rho_{j}}(m_{1}),\dotsc,\deg_{\rho_{j}}(m_{t})\}$. We write $\us:=(s_{1},\dotsc,s_{r})$, and for any $c-$dimensional cone $\sigma\!=\!\cone(\rho_{i_{1}},\dotsc, $ $\rho_{i_{c}})$, we set
\[
\cC^{\sigma}_{I}:=\cC^{\sigma}_{\us}
=\{m\in M \mid \langle m, \rho_{i_{p}}\rangle \geq s_{i_{p}} 
,\;1\leq p\leq c\}=\bigcap_{j=1}^{c}\cC^{\rho_{i_{j}}}_{I}.
\]
Next, we construct $\Delta^{\sigma}_{I}$. First, for any subset of monomials $\cS = \{n_1,\hdots,n_{s}\} \subset \{m_1,\hdots, m_{t}\}$ with $0 \leq s \leq t$, we define $\Delta^{\sigma}_{I}(
\cS)\subset \cC^{\sigma}_{I}
$
 recursively on $s$:
 
 If $s = 0$, then $\cS = \emptyset$ and we set:
\[
\Delta^{\sigma}_{I}(
\emptyset) := \{m \in M \mid s_{i_1}
 \leq \langle m,\rho_{i_1} \rangle, \hdots, s_{i_c}
 \leq \langle m,\rho_{i_c} \rangle\}.
\]

Otherwise, $s\geq 1$ and there is a permutation $\epsilon_{i_c} \in \mathfrak{S}_s$ such that
\[
\deg_{\rho_{i_c}}(n_{\epsilon_{i_c}(1)}) \leq \deg_{\rho_{i_c}}(n_{\epsilon_{i_c}(2)}) \leq \cdots \leq \deg_{\rho_{i_c}}(n_{\epsilon_{i_c}(s)}).
\]

\begin{itemize}
\item If $c = 1$ (and thus $\sigma$ is a ray), then 
\[
\Delta^{\sigma}_{I}(
\cS) := 
\{m \in M \mid s_{i_1}
\leq \langle m,\rho_{i_{1}} \rangle < \deg_{\rho_{i_1}}(n_{\epsilon_{i_1}(1)})
\}.
\]

\item Otherwise, $\Delta^{\sigma}_{I}(
\cS) := \bigcup_{j=0}^{s} \Delta^{\sigma}_{I}(
\cS)_{j}$ where: 

\[\begin{array}{l@{}}
\Delta^{\sigma}_{I}(
\cS)_{0} := \{m \in M \mid s_{i_{c}}
\leq \langle m, \rho_{i_{c}} \rangle < \deg_{\rho_{i_{c}}}(n_{\epsilon_{i_{c}}(1)})
\} \cap \Delta^{\sigma'}_{I}(
\emptyset),\\[0.5cm]

\Delta^{\sigma}_{I}(
\cS)_{j} := \{m \in M \mid \deg_{\rho_{i_{c}}}(n_{\epsilon_{i_{c}}(j)})
\leq \langle m, \rho_{i_{c}} \rangle < \deg_{\rho_{i_{c}}}(n_{\epsilon_{i_{c}}(j+1)}))
\}\; \cap\\[0.3cm]
\hfill \!\!\Delta^{\sigma'}_{I}(
\{n_{\epsilon_{i_{c}}(1)},\hdots, n_{\epsilon_{i_{c}}(j)}\}), \; 1 \leq j \leq s-1,\\[0.5cm]

\Delta^{\sigma}_{I}(
\cS)_{s} := \{m \in M \mid \deg_{\rho_{i_{c}}}(n_{\epsilon_{i_{c}}(s)}) 
\leq \langle m, \rho_{i_{c}} \rangle\} \!\cap\! \Delta^{\sigma'}_{I}(
\{n_{\epsilon_{i_{c}}(1)},\hdots, n_{\epsilon_{i_{c}}(s)}\}),\\[0.5cm]

\text{with}\;\; \sigma' = \cone(\rho_{i_1},\hdots, \rho_{i_{c-1}}).
\end{array}\]

\end{itemize}

\noindent Finally, we define $\Delta^{\sigma}_{I}
:= \Delta^{\sigma}_{I}(
\{m_{1},\hdots, m_{t}\})$.

\begin{definition}\label{Def:Klyachko diagram}
We call the collection of pairs $\{(\cC^{\sigma}_{I},\Delta^{\sigma}_{I})\}_{\sigma\in\Sigma}$ the {\em Klyachko diagram} of $I$. 
\end{definition}

Observe that each of the pairs $(\cC^{\sigma}_{I},\Delta^{\sigma}_{I})$ depicts a staircase diagram of the inclusion $I^{\sigma}\subset R^{\sigma}$ (see also Example \ref{Example:Klyachko diagram algorithm} below). Precisely, we have the following proposition showing that the Klyachko diagram characterizes the $\Sigma-$family of $I$.

\begin{proposition} \label{Prop:From ideal to diagram}
Let $I = (m_1,\hdots, m_{t}) \subset R$ be a monomial ideal with $m_{i}=\ux^{\uk^{i}}$ for $\uk^{i}\in\ZZ^{r}_{\geq0}$ and $1\leq i \leq t$. Let $\{(\cC^{\sigma}_{I},\Delta^{\sigma}_{I})\}_{\sigma\in\Sigma}$ be the Klyachko diagram of $I$. Then, for any 
cone $\sigma \in \Sigma$, 
\[
I^{\sigma}_{m}
\cong
\left\{
\begin{array}{@{}ll}
\bI, &m\in\cC^{\sigma}_{I}
\setminus\Delta^{\sigma}_{I}
\\
0  , &\text{otherwise.}
\end{array}
\right.
\]
In particular it holds
\begin{equation}\label{Eq:Identity KlyFilt2Dec}
\bigcup_{i=1}^{t}\cC^{\sigma}_{\uk^{i}}=\cC^{\sigma}_{I}\setminus\Delta^{\sigma}_{I}\quad \text{and}\quad \cC^{\sigma}_{0}\setminus \bigcup_{i=1}^{t}\cC^{\sigma}_{\uk^{i}}=\Delta^{\sigma}_{I}\cup(\cC^{\sigma}_{0}\setminus\cC^{\sigma}_{I}).
\end{equation}
\end{proposition}

 Before proving this result let us see an example that illustrate Proposition \ref{Prop:From ideal to diagram}, and shows how to compute the Klyachko diagram of a monomial ideal.

\begin{example}\label{Example:Klyachko diagram algorithm}
Let $R=\CC[x_{0},x_{1},x_2]$ be the Cox ring of $\PP^{2}$ with fan $\Sigma$ as in Example \ref{Example:Projective space and Hirzebruch}(i), and let $I=(x_{2}^2,x_{0}x_{2},x_{0}x_{1})$ be the monomial ideal of Example \ref{Example:monomial ideal}. First, we compute $s_{0}=s_{1}=s_{2}=0$ and we have 
\[
\begin{array}{ll}
\cC^{\rho_{0}}_{I}=\{(d_1,d_2)\mid d_1+d_2\leq 0\}&
\cC^{\rho_{1}}_{I}=\{(d_1,d_2)\mid d_1\geq0\}\\
\cC^{\rho_2}_{I}=\{(d_1,d_2)\mid d_2\geq0\}.
\end{array}
\]
We compute $\Delta^{\sigma_{0}}_{I}$. We order the monomials with respect to $\rho_2$: $\deg_{\rho_{2}}(x_{0}x_1)=0<\deg_{\rho_{2}}(x_{0}x_2)=1<\deg_{\rho_{2}}(x_2^2)=2$. We obtain
\[
\begin{array}{l}
\Delta^{\sigma_{0}}_{I}(\{x_{2}^2,x_{0}x_{2},x_{0}x_{1}\})_{0}=\emptyset\\
\Delta^{\sigma_{0}}_{I}(\{x_{2}^2,x_{0}x_{2},x_{0}x_{1}\})_{1}=\{(d_1,d_2)\mid d_2=0\}\cap \Delta^{\rho_1}(\{x_0x_1\})\}\\
\Delta^{\sigma_{0}}_{I}(\{x_{2}^2,x_{0}x_{2},x_{0}x_{1}\})_{2}=\{(d_1,d_2)\mid d_2=1\}\cap \Delta^{\rho_1}(\{x_0x_1,x_0x_2\})\\
\Delta^{\sigma_{0}}_{I}(\{x_{2}^2,x_{0}x_{2},x_{0}x_{1}\})_{3}=\emptyset.
\end{array}
\]
Since $\Delta^{\rho_1}_{I}(\{x_0x_1\})=\{(d_1,d_2)\mid d_1 =0\}$ and $\Delta^{\rho_{1}}_{I}(\{x_0x_1,x_0x_2\})=\emptyset$, we get $\Delta^{\sigma_{0}}_{I}(\{x_{2}^2,x_{0}x_{2},x_{0}x_{1}\})_{1}=\{(0,0)\}$ and $\Delta^{\sigma_{0}}_{I}(\{x_{2}^2,x_{0}x_{2},x_{0}x_{1}\})_{2}=\emptyset$. Therefore, $\Delta^{\sigma_0}_{I}=\{(0,0)\}$. Similarly, for the remaining cones we compute $\Delta^{\sigma_1}_{I}=\{(0,0),(-1,1)\}$ and $\Delta^{\sigma_2}_{I}=\emptyset$. (See figures \ref{Fig:PP2} and \ref{Fig:PP2alt}).

\begin{figure}[h]
\begin{tikzpicture}
    [
    	scale=0.5,
    	x=1cm,
    	y=1cm,
        dot/.style={
        	circle, 
       		fill=black,
       		inner xsep=0, 
       		inner ysep=2},
        dot2/.style={
        	circle,
        	fill=white,
        	inner xsep=0, 
        	inner ysep=1.2},
        dot3/.style={
        	circle,
        	fill=gray,
        	inner xsep=0, 
        	inner ysep=1}
    ]
\foreach \x in {-3,...,3}{
	\foreach \y in {-3,...,3}{
		\node[dot3] at (\x,\y){};
	}
}
\foreach \x in {0,...,3}{
	\foreach \y in {0,...,3}{
        \node[dot] at (\x,\y){ };
    }
}
\node[dot2] at (0,0) {};
\end{tikzpicture}
\hfill
\begin{tikzpicture}
     [
    	scale=0.5,
    	x=1cm,
    	y=1cm,
        dot/.style={
        	circle, 
       		fill=black,
       		inner xsep=0, 
       		inner ysep=2},
        dot2/.style={
        	circle,
        	fill=white,
        	inner xsep=0, 
        	inner ysep=1.2},
        dot3/.style={
        	circle,
        	fill=gray,
        	inner xsep=0, 
        	inner ysep=1}
    ]
\foreach \x in {-3,...,3}{
	\foreach \y in {-3,...,3}{
		\node[dot3] at (\x,\y){};
	}
}
\foreach \y in {0,...,3}{
	\foreach \x in {-3,...,-\y}{
		\node[dot] at (\x,\y){};
	}
}
\node[dot2] at (0,0) {};
\node[dot2] at (-1,1) {};
\end{tikzpicture}
\hfill
\begin{tikzpicture}
    [
    	scale=0.5,
    	x=1cm,
    	y=1cm,
        dot/.style={
        	circle, 
       		fill=black,
       		inner xsep=0, 
       		inner ysep=2},
        dot2/.style={
        	circle,
        	fill=white,
        	inner xsep=0, 
        	inner ysep=1.2},
        dot3/.style={
        	circle,
        	fill=gray,
        	inner xsep=0, 
        	inner ysep=1}
    ]
\foreach \x in {-3,...,3}{
	\foreach \y in {-3,...,3}{
		\node[dot3] at (\x,\y){};
	}
}
\foreach \x in {0,...,3}{
	\foreach \y in {-3,...,-\x}{
		\node[dot] at (\x,\y){};
	}
}
\end{tikzpicture}
\caption{$\circ$ stand for points of $\Delta^{\sigma_{0}}_{I}$ (respectively $\Delta^{\sigma_{1}}_{I}$ and $\Delta^{\sigma_{2}}_{I}$) inside the points $\bullet$ of $\cC^{\sigma_{0}}_{I}$ (respectively $\cC^{\sigma_{1}}_{I}$ and $\cC^{\sigma_{2}}_{I}$).}\label{Fig:PP2}
\end{figure}
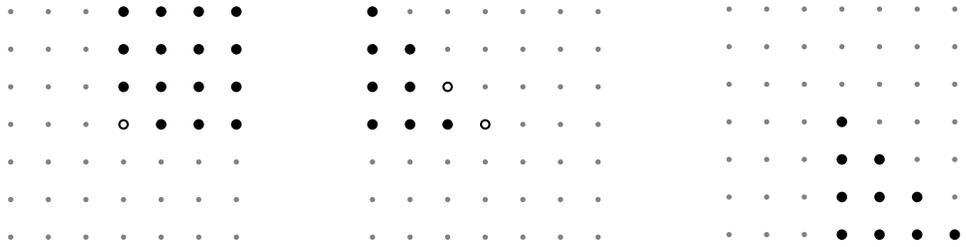

\begin{figure}[h]
\begin{tikzpicture}
    [
    	scale=0.7,
    	x=1cm,
    	y=1cm,
        dot/.style={
        	circle, 
       		fill=black,
       		inner xsep=0, 
       		inner ysep=2},
        dot2/.style={
        	circle,
        	fill=white,
        	inner xsep=0, 
        	inner ysep=1.2},
        dot3/.style={
        	circle,
        	fill=gray,
        	inner xsep=0, 
        	inner ysep=1}
    ]
\foreach \x in {-4,...,3}{
	\foreach \y in {-3,...,4}{
		\node[dot3] at (\x,\y){};
	}
}
\path[draw] (0,4) -- (0,1) -- (1,1) -- (1,0) -- (3,0);
\path[fill=gray, fill opacity =0.1] (0,4) -- (0,1) -- (1,1) -- (1,0) -- (3,0) -- (3,4);

\path[draw] (-4,4) -- (-2,2) -- (-3,2) -- (-1,0) -- (-4,0);
\path[fill=gray, fill opacity =0.1] (-4,4) -- (-2,2) -- (-3,2) -- (-1,0) -- (-4,0);

\path[draw] (0,-3) -- (0,0) -- (3,-3);
\path[fill=gray, fill opacity =0.1] (0,-3) -- (0,0) -- (3,-3);

\end{tikzpicture}
\caption{The Klyachko diagram of Figure \ref{Fig:PP2} represented together in a single figure. The shadowed region corresponds to each set $\cC^{\sigma_{i}}_{I}\setminus \Delta^{\sigma_{i}}_{I}$ for $i=0,1,2$.}\label{Fig:PP2alt}
\end{figure}
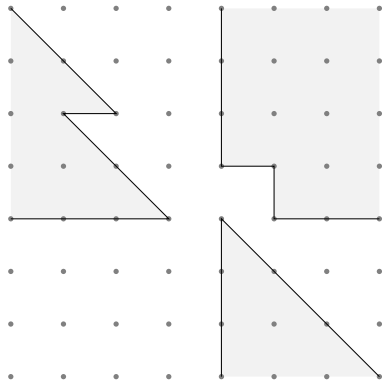

By Proposition \ref{Prop:From ideal to diagram}, we have
\[
\begin{array}{l@{\hspace{2mm}}l@{}}
I^{\rho_0}_{m}\cong
	\left\{
\begin{array}{@{}l@{\,}l@{}}
\bI, &m\!\in\!\{d_1+d_2\leq 0\}\\
0  , &\text{\hspace{4mm}otherwise.}
\end{array}
	\right.
&
I^{\rho_1}_{m}\cong
	\left\{
\begin{array}{@{}l@{\,}l@{}}
\bI, &m\!\in\!\{d_1\geq 0\}\\
0  , &\text{\hspace{4mm}otherwise.}
\end{array}
	\right.\\[3mm]
I^{\rho_2}_{m}\cong
	\left\{
\begin{array}{@{}l@{\,}l@{}}
\bI, &m\!\in\!\{d_2\geq 0\}\\
0  , &\text{\hspace{4mm}otherwise.}
\end{array}
	\right.
&
I^{\sigma_0}_{m}\cong
	\left\{
\begin{array}{@{}l@{\,}l@{}}
\bI, &m\!\in\!
\left\{ 
	\begin{array}{@{}l@{}}
		d_1\geq 0\\
		d_2\geq0
	\end{array}	
	\right\}\!\!\setminus\!\!\{(0,0)\}\\
0  , &\text{\hspace{4mm}otherwise.}
\end{array}
	\right.\\
I^{\sigma_1}_{m}\cong
	\left\{
\begin{array}{@{}l@{\,}l@{}}
\bI, &m\!\in\!
\left\{
	\begin{array}{@{}l@{}}
		d_1+d_2\leq 0\\
		d_2\geq0
	\end{array}	
	\right\}\!\!\setminus\!\!\left\{
		\begin{array}{@{}l@{}}
		(0,0)\\
		(-1,1)
		\end{array}\right\}\\[3mm]
0  , &\text{\hspace{4mm}otherwise.}
\end{array}
	\right.
&
I^{\sigma_2}_{m}\cong
	\left\{
\begin{array}{@{}l@{\,}l@{}}
\bI, &m\!\in\!
\left\{
	\begin{array}{@{}l@{}}
		d_1\geq 0\\
		d_1+d_2\leq0
	\end{array}	
	\right\}\\
0  , &\text{\hspace{4mm}otherwise}
\end{array}
	\right.
\end{array}
\]
which coincides with the $\Sigma-$family computed in Example \ref{Example:monomial ideal}.
\end{example}

\begin{proof}[Proof of Proposition \ref{Prop:From ideal to diagram}]
First, we recall that for $1\leq j\leq t$,
\[
R^{\rho_{j}}_{m}
=
\left\{
\begin{array}{@{}ll}
\CC\langle x_{1}^{\langle m, n(\rho_{1})\rangle
}
\dotsb x_{r}^{\langle m, n(\rho_{r})\rangle 
}\rangle,&\langle m, n(\rho_{j})\rangle\geq
0
\\
0,&\text{otherwise.}
\end{array}
\right.
\]
If $\sigma = \rho_{j}$ is a ray, then we have
\[
I^{\rho_{j}}_{m}
\cong
\left\{
\begin{array}{@{}ll}
\bI, &\langle m, n(\rho_{j})\rangle \geq s_{j}
\\
0  , &\langle m, n(\rho_{j})\rangle <    s_{j}
.
\end{array}
\right.
\]
Otherwise, $\sigma = \cone(\rho_{i_{1}}, \hdots, \rho_{i_{c}})$ for some $2 \leq c \leq r$. Assume that $m \in \Delta^{\sigma}_{I}(
\{m_1,\hdots,m_{t}\})_{j}$ for some $1 \leq j \leq t$ and $\epsilon_{i_{c}} \in \mathfrak{S}_{t}$ as above. Therefore, 
\[
\deg_{\rho_{i_{c}}}(m_{\epsilon_{i_{c}}(j)})
\leq \langle m, \rho_{i_{c}} \rangle < \deg_{\rho_{i_{c}}}(m_{\epsilon_{i_{c}}(j+1)}))
.
\]
In particular, if $j < t$, we have $R_{m}^{\sigma}(
-\uk^{\epsilon_{i_c}(j+1)}) = \cdots = R_{m}^{\sigma}(
-\uk^{\epsilon_{i_c}(t)}) = 0$. It suffices to see that $R_{m}^{\sigma}(
-\uk^{\epsilon_{i_c}(p)}) = 0$ for $1 \leq p \leq j$. On the other hand, $m \in \Delta^{\sigma'}_{I}(
\{m_{\epsilon_{i_{c}}(1)},\hdots, m_{\epsilon_{i_{c}}(j)}\})$. We repeat the same argument for $\sigma' = \cone(\rho_{i_{1}},\hdots, \rho_{i_{c-1}})$ and the set of monomials $\{m_{\epsilon_{i_{c}}(1)},\hdots, m_{\epsilon_{i_{c}}(j)}\}$, and so on. This procedure stops either when $\dim(\sigma') = 1$ or when the set of monomials is empty. In the latter case, it follows straightforward that we arrived at $R_{m}^{\sigma}(
-\uk^{\epsilon_{i_1}(1)}) = \cdots = R_{m}^{\sigma}(
-\uk^{\epsilon_{i_1}(t)}) = 0$ and hence $I^{\sigma}_{m}=0$. In the case $\dim(\sigma') = 1$, assume the set of monomials is $\{n_{\epsilon_{i_{1}}(1)},\hdots, n_{\epsilon_{i_{1}}(s)}\} \subset \{m_{1},\hdots, m_{t}\}$ with $s\leq j$. Since we have
\[
s_{i_1} 
\leq \langle m,\rho_{i_{1}} \rangle < \deg_{\rho_{i_1}}(n_{\epsilon_{i_1}(1)})
,
\]
then $R_{m}^{\sigma}(
-\uk^{\epsilon_{i_1}(1)}) = \cdots = R_{m}^{\sigma}(
-\uk^{\epsilon_{i_1}(s)}) = 0$, and we obtain $I^{\sigma}_{m}
 = 0$. Analogously, if $m \in \cC^{\sigma}_{I}
 \setminus 
\Delta^{\sigma}_{I}
$ we get $R^{\sigma}_{m}(
-\uk^{p}) \neq 0$ for some $1 \leq p \leq t$, so  $I^{\sigma}_{m}
 \cong \bI$.

\vspace{2mm}
\noindent Finally, (\ref{Eq:Identity KlyFilt2Dec}) follows from a comparison with the description of $I^{\sigma}_{m}$ in Proposition~\ref{Prop:KlyFilt for monomomial ideal}.
\end{proof}

Combining Propositions \ref{Prop:KlyFilt for monomomial ideal} and \ref{Prop:From ideal to diagram}, the next result shows how to obtain the Klyachko diagram of the sum of two monomial ideals.

\begin{corollary}\label{Corollary:KlyDec of sum}
Let $\{(\cC^{\sigma}_{I},\Delta^{\sigma}_{I})\}$ and $\{(\cC^{\sigma}_{J},\Delta^{\sigma}_{J})\}$ be the Klyachko diagrams of two monomial ideals $I$ and $J$, respectively. Then, the Klyachko diagram of $I+J$ is given by
\[
\left\{
\begin{array}{@{}l@{\,}l@{\,}l@{}}
\cC^{\sigma}_{I+J}&=&\{m\in M\mid \langle m,\rho\rangle \geq \min\{s_{\rho}^{I},s_{\rho}^{J}\},\ \rho\in\sigma(1)\}\\
\Delta^{\sigma}_{I+J}&
=&(\Delta^{\sigma}_{I}\cap\Delta^{\sigma}_{J})\cup
(\Delta_{I}\cap(\cC^{\sigma}_{I+J}\setminus\cC^{\sigma}_{J}))\cup
(\Delta_{J}\cap(\cC^{\sigma}_{I+J}\setminus\cC^{\sigma}_{I}))\cup\\
&&(\cC^{\sigma}_{I+J}\setminus(\cC^{\sigma}_{I}\cup\cC^{\sigma}_{J})).
\end{array}
\right.
\]
\end{corollary}
\begin{proof}
We write $I=(\ux^{\uk^{1}},\dotsc,\ux^{\uk^{t}})$ and $J=(\ux^{\ul^{1}},\dotsc,\ux^{\ul^{s}})$. Then, $s^{I}_{j}=\min\{k^{1}_{j},\dotsc,k^{t}_{j}\}$, $s^{J}_{j}=\min\{l^{1}_{j},\dotsc,l^{s}_{j}\}$ and $I+J=(\ux^{\uk^{1}},\dotsc,\ux^{\uk^{t}},\ux^{\ul^{1}},\dotsc,\ux^{\ul^{s}})$. It follows that $s^{I+J}_{j}=\min\{s^{I}_{j},s^{J}_{j}\}$ and then 
\[
\cC^{\sigma}_{I+J}=\{m\in M\mid \langle m,\rho\rangle \geq \min\{s_{\rho}^{I},s_{\rho}^{J}\},\ \rho\in\sigma(1)\}.
\]
In particular, $\cC^{\sigma}_{I}$, $\cC^{\sigma}_{J}$ are contained in $\cC^{\sigma}_{I+J}$. By Propositions \ref{Prop:KlyFilt for monomomial ideal} and \ref{Prop:From ideal to diagram} we have $\cC^{\sigma}_{I+J}\setminus\Delta^{\sigma}_{I+J}=(\cC^{\sigma}_{I}\setminus\Delta^{\sigma}_{I})\cup
(\cC^{\sigma}_{J}\setminus\Delta^{\sigma}_{J}).$ Taking complementaries with respect to $\cC^{\sigma}_{I+J}$ it yields
\[
\begin{array}{@{}l@{}l@{}}
\Delta^{\sigma}_{I+J}
&=\cC^{\sigma}_{I+J}\setminus ((\cC^{\sigma}_{I}\setminus\Delta^{\sigma}_{I})\cup
(\cC^{\sigma}_{J}\setminus\Delta^{\sigma}_{J}))\\
&=(\cC^{\sigma}_{I+J}\setminus (\cC^{\sigma}_{I}\setminus \Delta^{\sigma}_{I}))\cap
(\cC^{\sigma}_{I+J}\setminus(\cC^{\sigma}_{J}\setminus \Delta^{\sigma}_{J}))\\
&=(\Delta^{\sigma}_{I}\cup(\cC^{\sigma}_{I+J}\setminus\cC^{\sigma}_{I}))\cap
(\Delta^{\sigma}_{J}\cup(\cC^{\sigma}_{I+J}\setminus\cC^{\sigma}_{J})),
\end{array}
\]
and the result follows.
\end{proof}

The following example illustrates Corollary \ref{Corollary:KlyDec of sum}.

\begin{example}
Let $R=\CC[x_{0},x_{1},x_2]$ be the Cox ring of $\PP^{2}$ with fan $\Sigma$ as in Example \ref{Example:Projective space and Hirzebruch}(i), and let $I=(x_{1}^{2}x_{2}^{4},x_{1}^{3}x_{2},x_{1}^{5})$ and $J=(x_{2}^{4},x_{1}x_{2}^{3},x_{1}^{4}x_{2}^{2})$ be two monomial ideals. Notice that $(s_{0}^{I},s_{1}^{I},s_{2}^{I})=(0,2,0)$ and $(s_{0}^{J},s_{1}^{J},s_{2}^{J})=(0,0,2)$, so $(s_{0}^{I+J},s_{1}^{I+J},s_{2}^{I+J})=(0,0,0)$. Hence, $\cC^{\sigma_{0}}_{I}=\{d_{1}\geq 2,d_{2}\geq 0\}$, 
$\cC^{\sigma_{0}}_{J}=\{d_{1}\geq 0,d_{2}\geq 2\}$ and $\cC^{\sigma_{0}}_{I+J}=\{d_{1}\geq 0,d_{2}\geq 0\}$; while $\cC^{\sigma_{1}}_{I}=\cC^{\sigma_{1}}_{J}=\cC^{\sigma_{1}}_{I+J}=\{d_{1}+d_{2}\leq 0,d_{2}\geq0\}$ and
$\cC^{\sigma_{2}}_{I}=\cC^{\sigma_{2}}_{J}=\cC^{\sigma_{2}}_{I+J}=\{d_{1}+d_{2}\leq 0,d_{1}\geq0\}$.

Computing the remaining Klyachko diagrams of $I$ and $J$ we obtain:
\[
\begin{array}{ll}
\Delta^{\sigma_{0}}_{I}=\{(2, 0), (2, 1), (2, 2), (2, 3), (3, 0), (4, 0)\},&\Delta^{\sigma_{1}}_{I}=\Delta^{\sigma_{2}}_{I}=\emptyset\\
\Delta^{\sigma_{0}}_{J}=\{(0, 2), (0, 3), (1, 2), (2, 2), (3, 2)\},&\Delta^{\sigma_{1}}_{J}=\Delta^{\sigma_{2}}_{J}=\emptyset.
\end{array}
\]
Thus, applying Corollary \ref{Corollary:KlyDec of sum} we get $\Delta^{\sigma_{1}}_{I+J}=\Delta^{\sigma_{2}}_{I+J}=\emptyset$ and
\[
\begin{array}{@{}l@{}l}
\Delta^{\sigma_{0}}_{I+J}=\{&
(0, 0), (0, 1), (0, 2), (0, 3), (1, 0), (1, 1), (1, 2), (2, 0), (2, 
1), (2, 2),\\
&(3, 0), (4, 0)\}.
\end{array}
\]
Figure \ref{Fig:Example sum} illustrates this example.
\end{example}

\begin{figure}[h]
\begin{tikzpicture}
    [
    	scale=0.5,
    	x=1cm,
    	y=1cm,
        dot/.style={
        	circle, 
       		fill=black,
       		inner xsep=0, 
       		inner ysep=2},
        dot2/.style={
        	circle,
        	fill=white,
        	inner xsep=0, 
        	inner ysep=1.2},
        dot3/.style={
        	circle,
        	fill=gray,
        	inner xsep=0, 
        	inner ysep=1}
    ]
\foreach \x in {-1,...,6}{
	\foreach \y in {-1,...,6}{
		\node[dot3] at (\x,\y){};
	}
}
\path[draw,dashed] (0,6) -- (0,-1);
\path[draw,dashed] (-1,0) -- (6,0);
\path[draw, dotted] (2,6) -- (2,0) -- (2,6);
\path[draw] (2,6) -- (2,4) -- (3,4) -- (3,1) -- (5,1) --(5,0) -- (6,0);
\path[fill=gray, fill opacity =0.1] (2,6) -- (2,4) -- (3,4) -- (3,1) -- (5,1) --(5,0) -- (6,0) --(6,6);
\end{tikzpicture}
\hfill
\begin{tikzpicture}
    [
    	scale=0.5,
    	x=1cm,
    	y=1cm,
        dot/.style={
        	circle, 
       		fill=black,
       		inner xsep=0, 
       		inner ysep=2},
        dot2/.style={
        	circle,
        	fill=white,
        	inner xsep=0, 
        	inner ysep=1.2},
        dot3/.style={
        	circle,
        	fill=gray,
        	inner xsep=0, 
        	inner ysep=1}
    ]
\foreach \x in {-1,...,6}{
	\foreach \y in {-1,...,6}{
		\node[dot3] at (\x,\y){};
	}
}
\path[draw,dashed] (0,6) -- (0,-1);
\path[draw,dashed] (-1,0) -- (6,0);
\path[draw, dotted] (0,6) -- (0,2) -- (6,2);
\path[draw] (0,6) -- (0,4) -- (1,4) -- (1,3) -- (4,3) -- (4,2) --(6,2);
\path[fill=gray, fill opacity =0.1] (0,6) -- (0,4) -- (1,4) -- (1,3) -- (4,3) -- (4,2) --(6,2) --(6,6);
\end{tikzpicture}
\hfill
\begin{tikzpicture}
    [
    	scale=0.5,
    	x=1cm,
    	y=1cm,
        dot/.style={
        	circle, 
       		fill=black,
       		inner xsep=0, 
       		inner ysep=2},
        dot2/.style={
        	circle,
        	fill=white,
        	inner xsep=0, 
        	inner ysep=1.2},
        dot3/.style={
        	circle,
        	fill=gray,
        	inner xsep=0, 
        	inner ysep=1}
    ]
\foreach \x in {-1,...,6}{
	\foreach \y in {-1,...,6}{
		\node[dot3] at (\x,\y){};
	}
}
\path[draw,dashed] (0,0) -- (0,-1);
\path[draw,dashed] (-1,0) -- (0,0);
\path[draw, dotted] (0,6) -- (0,0) -- (6,0);
\path[draw] (0,6) -- (0,4) -- (1,4) -- (1,3) -- (3,3)  -- (3,1) -- (5,1) --(5,0) -- (6,0);
\path[fill=gray, fill opacity =0.1] (0,6) -- (0,4) -- (1,4) -- (1,3) -- (3,3)  -- (3,1) -- (5,1) --(5,0) -- (6,0) -- (6,6);
\end{tikzpicture}

\caption{The part of the Klyachko diagram associated to the cone $\sigma_{0}$ of $I$, $J$ and $I+J$ respectively. The dotted part corresponds to $\cC^{\sigma_{0}}_{\bullet}$ and the shadowed part to $\cC^{\sigma_{0}}_{\bullet}\setminus \Delta^{\sigma_{0}}_{\bullet}$.}\label{Fig:Example sum}
\end{figure}
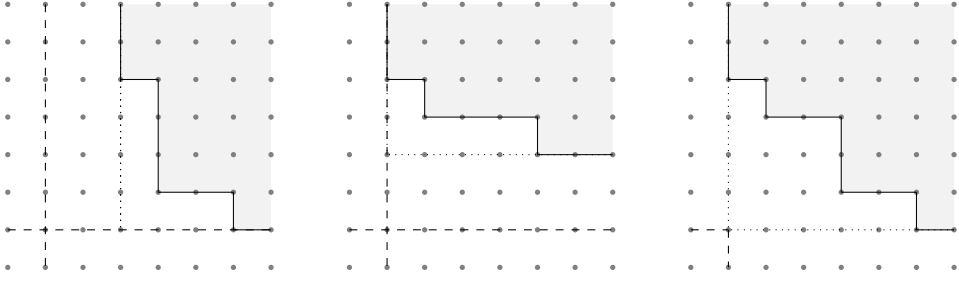

We end this subsection with more examples on the compution of Klyachko diagrams.
\begin{example}
Let $R=\CC[x_0,x_1,x_2,x_3]$ be the Cox ring of $\PP^{3}$ with fan $\Sigma$ as in Example \ref{Example:Projective space and Hirzebruch}(i), and let $I=(x_0x_1,x_1x_2x_3^2,x_2^2)$ be a monomial ideal. We have $s_0=s_1=s_2=s_3=0$ and 
\[
\begin{array}{ll}
\cC^{\rho_{0}}_{I}=\{(d_1,d_2,d_3)\mid d_1+d_2+d_3\leq 0\}&
\cC^{\rho_{1}}_{I}=\{(d_1,d_2,d_3)\mid d_1\geq0\}\\
\cC^{\rho_2}_{I}=\{(d_1,d_2,d_3)\mid d_2\geq0\}&
\cC^{\rho_3}_{I}=\{(d_1,d_2,d_3)\mid d_3\geq0\}.
\end{array}
\]
We compute $\Delta^{\sigma_0}_{I}$. We order the monomials with respect to $\rho_3$: $\deg_{\rho_3}(x_0x_1)=\deg_{\rho_3}(x_2^2)=0<\deg_{\rho_{3}}(x_1x_2x_3^2)=2$, and we obtain
\[
\begin{array}{l}
\Delta^{\sigma_{0}}_{I}(\cG)_{0}=\emptyset\\
\Delta^{\sigma_{0}}_{I}(\cG)_{1}=\emptyset\\
\Delta^{\sigma_{0}}_{I}(\cG)_{2}=\{(d_1,d_2,d_3)\mid 0\leq d_3< 2\}\cap \Delta^{\sigma_0'}_{I}(\{x_0x_1,x_2^2\})\\
\Delta^{\sigma_{0}}_{I}(\cG)_{3}=\{(d_1,d_2,d_3)\mid 2\leq d_3\}\cap \Delta^{\sigma_0'}_{I}(\{x_0x_1,x_2^2, x_1x_2x_3^2\}),
\end{array}
\]
where $\cG=\{x_0x_1,x_1x_2x_3^2,x_2^2\}$ and $\sigma_0'=\cone(\rho_1,\rho_2)$. We proceed computing $\Delta^{\sigma_0'}_{I}(\{x_0x_1,x_2^2\})$, ordering the two monomials with respect to $\rho_2$: $\deg_{\rho_{2}}(x_0x_1)=0<\deg_{\rho_{2}}(x_2^2)=2$. We get
\[
\begin{array}{l}
\Delta^{\sigma_0'}_{I}(\{x_0x_1,x_2^2\})_0=\emptyset\\
\Delta^{\sigma_0'}_{I}(\{x_0x_1,x_2^2\})_1=\{(d_1,d_2,d_3)\mid 0\leq d_2 < 2\}\cap \Delta^{\rho_1}_{I}(\{x_0x_1\})\\
\Delta^{\sigma_0'}_{I}(\{x_0x_1,x_2^2\})_2=\{(d_1,d_2,d_3)\mid 2\leq d_2\}\cap \Delta^{\rho_1}_{I}(\{x_0x_1,x_2^2\}),
\end{array}
\]
and $\Delta^{\rho_1}_{I}(\{x_0x_1\})=\{(d_1,d_2,d_3)\mid d_1=0\}$, while $\Delta^{\rho_1}_{I}(\{x_0x_1,x_2^2\})=\emptyset$. Hence,
$\Delta^{\sigma_0'}_{I}(\{x_0x_1,x_2^2\})=\{(d_1,d_2,d_3)\mid 0\leq d_2 < 2,\,d_1=0\}$, and 
\[
\Delta^{\sigma_0}_{I}(\cG)_2=\{(d_1,d_2,d_3)\mid0\leq d_3<2,\,0\leq d_2 < 2,\,d_1=0\}.\]
Similarly, we obtain $\Delta^{\sigma_0'}_{I}(\{x_0x_1,x_2^2, x_1x_2x_3^2\})\!=\!\Delta^{\sigma_0'}_{I}(\{x_0x_1,x_2^2, x_1x_2x_3^2\})_1\cup$ $\Delta^{\sigma_0'}_{I}(\{x_0x_1,x_2^2, x_1x_2x_3^2\})_2=\{(d_1,d_2,d_3)\mid 0\leq d_2\leq1,\,d_1=0\}$. Hence,
\[
\Delta^{\sigma_0}_{I}(\cG)_3=\{(d_1,d_2,d_3)\mid2\leq d_3,\,0\leq d_2\leq1,\,d_1=0\},\;\text{and}
\]
\[\Delta^{\!\sigma_0}_{I}\!\!=\!\!\{(0,d_2,d_3)\mid 0\leq d_3\leq 1,\,0\leq d_2\leq 1\}\cup\{(0,d_2,d_3)\mid 2\leq d_3,\,0\leq d_2\leq 1\}.\]
Applying the same procedure for the remaining cones, we get
\[
\begin{array}{l}
\Delta^{\sigma_{1}}_{I}=\{(-d_2-d_3,d_2,d_3)\mid 0\leq d_2,d_3\leq 1\}\cup
\{(-d_3,0,d_3)\mid 2\leq d_3\}\\
\Delta^{\sigma_{2}}_{I}=\emptyset\\
\Delta^{\sigma_{3}}_{I}=\{(0,d_2,d_3)\mid0\leq d_2\leq 1,\,d_3\leq -d_2\}\cup
\{(d_1,0,-d_1)\mid 0\leq d_1\}.
\end{array}
\]
We notice that in this example $\Delta^{\sigma_{1}}_{I}$ and $\Delta^{\sigma_{3}}_{I}$ are unbounded.
\end{example}

\begin{example}
Let $R=\CC[x_0,x_1,y_0,y_1]$ be the Cox ring of the Hirzebruch surface $\cH_3$ with fan $\Sigma$ as in Example \ref{Example:Projective space and Hirzebruch}(ii). $R$ is endowed with a $\ZZ^2-$grading such that $\deg(x_0)=\deg(x_1)=(1,0)$, $\deg(y_0)=(0,1)$ and $\deg(y_1)=(-3,1)$. We consider the monomial ideal $I=(x_{1}, x_{0}^{3}y_{1})$, so $s_{\rho_{0}}=s_{\rho_{1}}=s_{\eta_{0}}=s_{\eta_{1}}=0$. Thus,
\[
\begin{array}{ll}
\cC^{\rho_{0}}_{I}=\{(d_1,d_2)\mid d_1\leq 3d_2\}&\cC^{\rho_{1}}_{I}=\{(d_1,d_2)\mid d_1\geq0\}\\
\cC^{\eta_{0}}_{I}=\{(d_1,d_2)\mid d_2\leq 0\} & \cC^{\eta_{1}}_{I}=\{(d_1,d_2)\mid d_2\geq0\}.
\end{array}
\]
We compute $\Delta^{\sigma_{00}}$. We order the monomials with respect to $\eta_{1}$: $\deg_{\eta_{1}}(x_{1})=0<\deg_{\eta_{1}}(x_{0}^{3}y_{1})=1$. We have
\[
\begin{array}{l}
\Delta^{\sigma_{00}}_{I}(\cG)_{0}=\emptyset\\
\Delta^{\sigma_{00}}_{I}(\cG)_{1}=\{(d_1,d_2)\mid d_2=0\}\cap \Delta^{\rho_1}_{I}(\{x_1\})\\
\Delta^{\sigma_{00}}_{I}(\cG)_{2}=\{(d_1,d_2)\mid d_2=1\}\cap \Delta^{\rho_1}_{I}(\{x_1,x_0^{3}y_{1}\}),
\end{array}
\]
where $\cG=\{x_{1},x_{0}^{3}y_{1}\}$. Since $\Delta^{\rho_{1}}_{I}(\{x_1\})=\{(d_1,d_2)\mid d_1\geq0\}$ and $\Delta^{\rho_{1}}_{I}(\{x_1,$ $x_0^{3}y_{1}\})=\emptyset$, we obtain that $\Delta^{\sigma_{00}}_{I}=\{(0,0)\}$. Applying the same procedure, we arrive at $\Delta^{\sigma_{01}}_{I}=\Delta^{\sigma_{10}}_{I}=\Delta^{\sigma_{11}}_{I}=\emptyset$. (See figure \ref{Fig:Hirzebruch}).
\end{example}

\begin{figure}[h]
\begin{tikzpicture}
    [
    	scale=0.5,
    	x=1cm,
    	y=1cm,
        dot/.style={
        	circle, 
       		fill=black,
       		inner xsep=0, 
       		inner ysep=2},
        dot2/.style={
        	circle,
        	fill=white,
        	inner xsep=0, 
        	inner ysep=1.2},
        dot3/.style={
        	circle,
        	fill=gray,
        	inner xsep=0, 
        	inner ysep=1}
    ]
\foreach \x in {-3,...,6}{
	\foreach \y in {-3,...,3}{
		\node[dot3] at (\x,\y){};
	}
}
\node[draw] at (0,0) {};
\foreach \x in {0,...,6}{
	\foreach \y in {0,...,3}{
        \node[dot] at (\x,\y){ };
    }
}
\node[dot2] at (0,0) {};
\end{tikzpicture}
\hfill
\begin{tikzpicture}
     [
    	scale=0.5,
    	x=1cm,
    	y=1cm,
        dot/.style={
        	circle, 
       		fill=black,
       		inner xsep=0, 
       		inner ysep=2},
        dot2/.style={
        	circle,
        	fill=white,
        	inner xsep=0, 
        	inner ysep=1.2},
        dot3/.style={
        	circle,
        	fill=gray,
        	inner xsep=0, 
        	inner ysep=1}
    ]
\foreach \x in {-4,...,9}{
	\foreach \y in {-3,...,3}{
		\node[dot3] at (\x,\y){};
	}
}
\foreach \y [evaluate=\y as \b using \y*3] in {0,...,3}{
	\foreach \x in {-4,...,\b}{
		\node[dot] at (\x,\y){};
	}
}
\node[draw] at (0,0) {};
\end{tikzpicture}

\vspace{5mm}
\begin{tikzpicture}
     [
    	scale=0.5,
    	x=1cm,
    	y=1cm,
        dot/.style={
        	circle, 
       		fill=black,
       		inner xsep=0, 
       		inner ysep=2},
        dot2/.style={
        	circle,
        	fill=white,
        	inner xsep=0, 
        	inner ysep=1.2},
        dot3/.style={
        	circle,
        	fill=gray,
        	inner xsep=0, 
        	inner ysep=1}
    ]
\foreach \x in {-6,...,4}{
	\foreach \y in {-2,...,3}{
		\node[dot3] at (\x,\y){};
	}
}
\foreach \y [evaluate=\y as \b using \y*3] in {-2,...,0}{
	\foreach \x in {\b,...,4}{
		\node[dot] at (\x,\y){};
	}
}
\node[draw] at (0,0) {};
\end{tikzpicture}
\hfill
\begin{tikzpicture}
    [
    	scale=0.5,
    	x=1cm,
    	y=1cm,
        dot/.style={
        	circle, 
       		fill=black,
       		inner xsep=0, 
       		inner ysep=2},
        dot2/.style={
        	circle,
        	fill=white,
        	inner xsep=0, 
        	inner ysep=1.2},
        dot3/.style={
        	circle,
        	fill=gray,
        	inner xsep=0, 
        	inner ysep=1}
    ]
\foreach \x in {-3,...,6}{
	\foreach \y in {-2,...,3}{
		\node[dot3] at (\x,\y){};
	}
}
\foreach \x in {0,...,6}{
	\foreach \y in {-2,...,0}{
        \node[dot] at (\x,\y){ };
    }
\node[draw] at (0,0) {};
}
\end{tikzpicture}
\caption{Klyachko diagram of Example \ref{Example:Klyachko diagram algorithm} (iii). It displays $\Delta^{\sigma_{ij}}_{I}$ ($\circ$) inside $\cC^{\sigma_{ij}}_{I}$ ($\bullet$) for $(i,j)=(0,0),$ $(1,0),$ $(0,1),$ $(1,1)$, clockwise. In each picture $\square$ places the origin $(0,0)\in M\cong\ZZ^{2}$.}\label{Fig:Hirzebruch}
\end{figure}
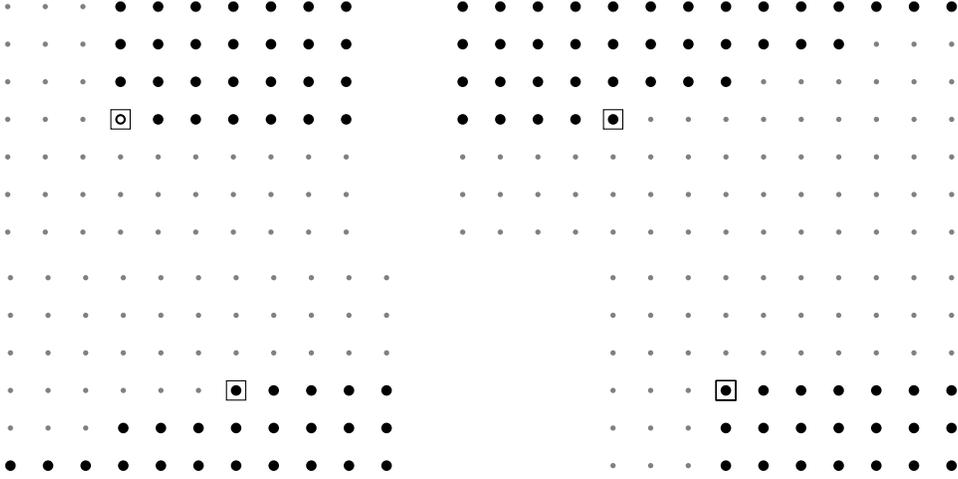
\subsection{From a Klyachko diagram to a monomial ideal}\label{Section:From Klyachko diagram to ideal}

\vspace{0.5cm}
Our next goal is to find the minimal set of monomials generating the saturated ideal $I$ associated to a Klyachko diagram $\{(\cC^{\sigma}_{I},\Delta^{\sigma}_{I})\}_{\sigma \in \Sigma}$. 
We may assume that $\Delta^{\sigma} \neq \emptyset$ for some cone $\sigma \in \Sigma$. Otherwise, by Proposition \ref{Prop:From ideal to diagram} and Lemma \ref{Lemma:KlyFilt for monomial} the $\Sigma-$family of $I$ would be the $\Sigma-$family of a principal monomial ideal. Then, $I=(x_{1}^{s_{1}}\dotsb x_{r}^{s_{r}})$, where $s_{i}\in\ZZ$ such that $\cC^{\rho_{j}}_{I}=\{m\in M\mid\langle m,\rho_{j}\rangle \geq s_{j}\}$.

Since $I$ is $\Cl(X)-$graded and finitely generated, the monomials minimally generating $I$ belong in a finite number of homogeneous pieces.  For any $D=(a_{\rho})_{\rho\in\Sigma(1)}\in\ZZ^{r}$, we denote by $[D]=\sum_{\rho\in\Sigma(1)}a_{\rho}[D_{\rho}]\in\Cl(X)$ the class of the corresponding Weil divisor in $X$. We first start by providing a monomial basis of $I_{[D]}\subset R_{[D]}$.

\begin{lemma}\label{Lemma:generators of I}
Let $\{(\cC^{\sigma}_{I},\Delta^{\sigma}_{I})\}_{\sigma\in\Sigma}$ be a Klyachko diagram of a $B-$satu\-rated monomial ideal $I$, and $D = (a_{\rho})_{\rho\in\Sigma(1)}\in\ZZ^{r}$. Then,
\[I_{[D]}=\CC\left\langle
\ux^{m+D}|m\in\bigcap_{\sigma\in\Sigma_{\max}}(\cC^{\sigma}_{I}(D)\setminus \Delta^{\sigma}_{I}(D))\right\rangle.\]
\end{lemma}
\begin{proof}
For any cone $\sigma\in\Sigma$ we have that $$I_{[D]}^{\sigma}=I^{\sigma}_{0}(D)=\bigoplus_{m\in M} I(D)_{m}, \text{ and } $$ $$R^{\sigma}_{[D]}=\CC\langle \ux^{m+D}\mid m\in\cC^{\sigma}_{0}(D)\rangle.$$
On the other hand, for any $m\in M$ we have 
\[
I(D)_{m}\cong H^{0}(X,\tilde{I}(D))_{m}\cong\bigcap_{\sigma\in\Sigma_{\max}}I^{\sigma}(D)_{m}.
\] By Proposition \ref{Prop:From ideal to diagram}, for any cone $\sigma\in\Sigma$, $I^{\sigma}_{m}(D)\neq0$ if and only if $m\in \cC^{\sigma}_{I}(D)\setminus \Delta^{\sigma}_{I}(D)$. Hence,
\[
I_{[D]}=\bigoplus_{m\in M} \bigcap_{\sigma\in\Sigma_{\max}}I^{\sigma}_{m}(D)\hspace{3mm}=
\hspace{-10mm}
\bigoplus_{
{m\in\hspace{-4mm}\displaystyle{
\bigcap_{\sigma\in\Sigma_{\max}}
}\hspace{-3mm}(\cC^{\sigma}_{I}(D)\setminus \Delta^{\sigma}_{I}(D))}
}
\hspace{-8mm} \CC\langle \ux^{m+D}\rangle
\]
and the lemma follows.
\end{proof}

The following remark shows how shifting by a multidegree $D\in\ZZ^{r}$ affects the Klyachko diagram.
\begin{remark}\label{Remark:KlyDec for a shift}
Since $X$ is smooth, for any $D=(a_{\rho})_{\rho\in\Sigma(1)}\in\ZZ^{r}$ and any cone $\sigma\in\Sigma$, there is a character $\tau_{\sigma}\in M$ such that $\langle \tau_{\sigma},\rho\rangle=a_{\rho}$ for all $\rho\in\sigma(1)$. Then, for any $m\in M$ and $\rho\in\sigma(1)$, $\langle m,\rho\rangle+a_{\rho}=\langle m+\tau_{\sigma},\rho\rangle$, so $R^{\sigma}_{m}(D)=R^{\sigma}_{m+\tau_{\sigma}}$. It follows from (\ref{Eq:Presentation monomial ideal}) that $I^{\sigma}_{m}(D)=I^{\sigma}_{m+\tau_{\sigma}}$. Therefore, the Klyachko diagram of the shifted monomial ideal $I(D)$ is given by
\[
\left\{
\begin{array}{@{}l@{}l@{}l}
\cC^{\sigma}_{I(D)}&=\cC^{\sigma}_{I}(D)&:=\cC^{\sigma}_{I}+\tau_{\sigma}\\
\Delta^{\sigma}_{I(D)}&=\Delta^{\sigma}_{I}(D)&:=\Delta^{\sigma}_{I}+\tau_{\sigma},
\end{array}
\right.
\]
which is obtained applying translations to the original Klyachko diagram.
\end{remark}

By Lemma \ref{Lemma:generators of I} we  already know a basis of each homogeneous piece of $I$. Our next task is to characterize which monomials in a homogeneous piece $I_{[E]}$ are divisible by a single monomial $x^{m+D}$. The following Lemma answers this question.

\begin{lemma}\label{Lemma:Span of a monomial}
Let $\{(\cC^{\sigma}_{I},\Delta^{\sigma}_{I})\}_{\sigma\in\Sigma}$ be a Klyachko diagram of a $B-$satu\-rated monomial ideal $I$, and $D = (a_{\rho})_{\rho\in\Sigma(1)}\in\ZZ^{r}$. Let $m\in\bigcap_{\Sigma_{\max}}(\cC^{\sigma}_{I}(D)\setminus \Delta^{\sigma}_{I}(D))$ and let $E=(b_{\rho})_{\rho\in\Sigma(1)}\in\ZZ^{r}$ be such that $b_{\rho}\geq a_{\rho}$ for all $\rho\in\Sigma(1)$. The set of monomials in $I_{[E]}$ which are divisible by $\ux^{m+D}$ is
\[
T_{E}(\ux^{m+D}):=\left\{\ux^{m'+E}\middle.\middle|
\langle m',\rho\rangle\geq\langle m,\rho\rangle+a_{\rho}-b_{\rho}, \rho\in\Sigma(1)\right\}.
\]
\end{lemma}
\begin{proof}
For any $m'\in M$, the monomial $\ux^{m'+E}$ is divisible by $\ux^{m+D}$ if and only if $\langle m',\rho\rangle+b_{\rho}\geq \langle m,\rho\rangle+a_{\rho}$ for any $\rho\in\Sigma(1)$, and the lemma follows.
\end{proof}

Now, let $\cG=\{\ux^{m_{1}+\uk^{1}},\dotsc,\ux^{m_{t}+\uk^{t}}\}$ be a finite set of monomials of possibly different degrees. For any $E=(b_{\rho})_{\rho\in\Sigma(1)}\in\ZZ^{r}$ such that $b_{\rho}\geq k^{i}_{\rho}$ for $\rho\in\Sigma(1)$ and $1\leq i\leq t$, we define 
\[
T_{E}\cG:=T_{E}(x^{m_{1}+\uk^{1}})\cup\dotsb\cup T_{E}(\ux^{m_{t}+\uk^{t}}),
\]
which describes the span of the monomials of $\cG$ inside $I_{[E]}$.

Finally, we can describe a finite set of generators of a $B-$saturated monomial ideal $I$ corresponding to a given Klyachko diagram.
Since $X$ is smooth we can assume that $\Cl(X)\cong\ZZ\langle [D_{\rho_{i_{1}}}],\dotsc,[D_{\rho_{i_{\ell}}}]\rangle\cong\ZZ^{\ell}$. Up to permutation of variables, we may also assume that $i_{1}=1,\dotsc,i_{\ell}=\ell$, and for any $a=(a_{1},\dotsc,a_{\ell})\in\ZZ^{\ell}$ we set $\overline{a}:=(a_{1},\dotsc,a_{\ell},0,\dotsc,0)\in\ZZ^{r}$. For any $u,v\in\ZZ^{\ell}$ we say that $v\preceq u$, if $u_{i}\geq v_{i}$
 for $1\leq i\leq \ell$, which defines a partial order.
We set $\cG_0 = \emptyset$ and for any $u\in\ZZ^{\ell}$, we define 
\[
\cG_{u} := \left\{\ux^{m+\overline{u}}\mid m\in\bigcap_{\sigma \in \Sigma_{\max}} \cC^{\sigma}_{I}(\overline{u})\setminus \Delta^{\sigma}_{I}(\overline{u})\right\}\setminus \bigcup_{v \preceq u} T_{\overline{u}}\cG_{v}
\]
assuming we have determined $\cG_{v}$ for any $v\preceq u$. Since $I$ is finitely generated, there are only finitely many degrees $u\in\ZZ^{\ell}$ such that $\cG_{u}\neq \emptyset$. 

If $R$ is the Cox ring of $\PP^{r-1}$, and so $R$ has the standard $\ZZ-$grading, then by construction that method gives directly a minimal set of monomial generators for $I$. The following example illustrates the method:

\begin{example}
(i) Let $R=\CC[x_{0},x_{1},x_2]$ be the Cox ring of $\PP^{2}$ (See Example \ref{Example:Projective space and Hirzebruch}(i)). Consider the following Klyachko diagram $\{ (\cC^{i}_{I},\Delta^{i}_{I})\}_{0\leq i\leq 2}$, where $\cC^{i}_{I}$ and $\Delta^{i}_{I}$ stand for $\cC^{\sigma_{i}}_{I}$ and $\Delta^{\sigma_{i}}_{I}$:
\[
\begin{array}{l}
\cC^{0}_{I}=\{(d_{1},d_{2})\mid d_{1}\geq0,\, d_{2}\geq0\}\\
\cC^{1}_{I}=\{(d_{1},d_{2})\mid d_{1}+d_{2}\leq0,\,d_{2}\geq0\}\\
\end{array}
\cC^{2}_{I}=\{(d_{1},d_{2})\mid d_{1}+d_{2}\leq0,\,d_{1}\geq0\},
\]
$\Delta^{0}_{I}=\{(0,0),(1,0)\}$ and  $\Delta^{1}_{I}=\Delta^{2}_{I}=\{(0,0)\}$. Applying the above procedure, we obtain that $\cG_{0}=\cG_{1}=\emptyset$, $\cG_{2}=\{x_{0}x_{2},x_1x_2\}$, $\cG_{3}=\{x_{0}x_{1}^{2}\}$ and $\cG_{j}=\emptyset$ for $j\geq 4$. Hence, the saturated monomial ideal corresponding to this Klyachko diagram is $I=(x_{0}x_{2},x_1x_2,x_{0}x_{1}^{2})$.

(ii) Let $R=\CC[x_{0},x_{1},x_2, x_{3}]$ be the Cox ring of $\PP^{3}$ (See Example \ref{Example:Projective space and Hirzebruch}(i)). Consider the following Klyachko diagram $\{(\cC^{i}_{I},\Delta^{i}_{I})\}_{0\leq i\leq 3}$, where $\cC^{i}_{I}$ and $\Delta^{i}_{I}$ stand for $\cC^{\sigma_{i}}_{I}$ and $\Delta^{\sigma_{i}}_{I}$:
\[
\begin{array}{l}
\cC^{0}_{I}=\{(d_{1},d_{2},d_{3})\mid d_{1}\geq0,\, d_{2}\geq0, d_{3} \geq 0\}\\
\cC^{1}_{I}=\{(d_{1},d_{2},d_{3})\mid d_{1}+d_{2}+d_{3}\leq0,\,d_{2}\geq0, d_{3}\geq0\}\\
\cC^{2}_{I}=\{(d_{1},d_{2},d_{3})\mid d_{1}+d_{2}+d_{3}\leq0,\,d_{1}\geq0,d_{3} \geq 0\}\\
\cC^{3}_{I}=\{(d_{1},d_{2},d_{3}) \mid d_{1}+d_{2}+d_{3}\leq0, \, d_{1}\geq 0,d_{2}\geq0\},
\end{array}
\]
$\Delta^{0}_{I}=\{(0,0,d_{3})\mid d_{3}\geq0\}\cup\{(0,1,0)\}$ and $\Delta^{1}_{I}=\Delta^{2}_{I}=\Delta^{3}_{I} = \emptyset$. Applying the above procedure, we obtain that $\cG_{0}= \emptyset$, $\cG_{1} = \{x_1\}$, $\cG_{2}=\{x_0^2x_1, x_1^2, x_1x_2, x_1x_3, x_2^2, x_2x_3\} \setminus T_{(2,0,0,0)}\{x_1\} = \{x_2^2,x_2x_3\}$ and $\cG_{j} = \emptyset$ for $j \geq 3$. Hence, the saturated monomial ideal corresponding to this Klyachko diagram is $I=(x_1,x_2^2,x_2x_3)$.

\end{example}
In more general gradings, we cannot assure that this method gives a {\em minimal} set of generators of $I$, but a finite set of monomials generating $I$. However, we can extract from it a minimal set of monomials generating $I$ by using suitable monomial divisions. The following example illustrates this situation:

\begin{example}
Let $R=\CC[x_{0},x_{1},y_{0},y_{1}]$ be the Cox ring of the Hirzebruch surface $\cH_{3}$ (see Example \ref{Example:Projective space and Hirzebruch}(ii)). In particular, $R$ is $\ZZ^{2}-$graded with $\deg(x_{0})=\deg(x_{1})=(1,0)$, $\deg(y_{0})=(0,1)$ and $\deg(y_{1})=(-3,1)$. Let $\{(\cC^{\sigma_{ij}}_{I},\Delta^{\sigma_{ij}}_{I})\}_{0\leq i,j\leq 1}$ be the Klyachko diagram of Example \ref{Example:Klyachko diagram algorithm} (iii).
Applying the above procedure we obtain that $\cG_{u}=\emptyset$ for any $u\in\ZZ^{2}\setminus\{(1,0), (0,1)\}$, and
$\cG_{(1,0)}=\{x_{1}\}$ and $\cG_{(0,1)}=\{x_0^3y_1,x_0^2x_1y_1, x_0x_1^2y_1,x_1^3y_1\}$. Hence $\{x_{1},x_0^3y_1,x_0^2x_1y_1, x_0x_1^2y_1,x_1^3y_1\}$ is a set of generators for a saturated monomial ideal $I$ corresponding to this Klyachko diagram. However, the first monomial divides the three last monomials. Therefore, $I$ is {\em minimally} generated by $\{x_{1}, x_{0}^{3}y_{1}\}$. 
\end{example}

\subsection{Non-saturated monomial ideals}\label{Section:HH1}

The previous subsections have shown that the theory of Klyachko diagrams is well suited to describe saturated monomial ideals, but we cannot retrieve directly information of non-saturated monomial ideals. In this subsection, we describe the quotient $I^{\sat}/I\cong H^{1}_{B}(I)$ using the Klyachko diagram $\{ (\cC^{\sigma}_{I},\Delta^{\sigma}_{I}) \} _{\sigma \in \Sigma}$ and the generators of $I$.

\begin{proposition}\label{Corollary:local HH1}
Let $I=(\ux^{m_{1}+\uk^{t}},\dotsc,\ux^{m_{t}+\uk^{t}})$ be a monomial ideal with Klyachko diagram $\{(\cC^{\sigma}_{I},\Delta^{\sigma}_{I})\} _{\sigma \in \Sigma}$, such that for $1\leq i\leq t$, $m_{i}\in M$ and $\uk^{i}=(k^{i}_{\rho})_{\rho\in\Sigma(1)}\in\ZZ^{r}$ satisfy $\langle m_{i},\rho\rangle+k^{t}_{\rho}\geq0$ for all $\rho\in\Sigma(1)$. Then, for any $D=(a_\rho)_{\rho\in\Sigma(1)}\in\ZZ^{r}$, 

\[
H_{B}^{1}(I)_{[D]}\!\cong\!
\CC\left\langle \!\left\{\ux^{m+D}\middle|\middle. m\in\hspace{-4mm}\bigcap_{\sigma\in\Sigma_{\max}}\hspace{-3mm}(\cC^{\sigma}_{I}(D)\setminus\Delta^{\sigma}_{I}(D))\!\right\}\!\setminus \bigcup_{i=1}^{t}T_{D}(\ux^{m_{i}+\uk^{i}})\!\right\rangle.
\]
\end{proposition}
\begin{proof}
From Lemma \ref{Lemma:Span of a monomial},
$
I_{[D]}=\bigcup_{i=1}^{t}T_{D}(\ux^{m_{i}+\uk^{i}}) \subset I^{sat}_{[D]}
$. By Proposition \ref{Prop:From ideal to diagram}, the Klyachko diagram characterizes the saturation of $I$, and by Lemma \ref{Lemma:generators of I} we have
$
I^{sat}_{[D]}=\CC\langle \ux^{m+D}|$ $ m\in\bigcap_{\sigma\in\Sigma_{\max}}(\cC^{\sigma}_{I}(D)\setminus\Delta^{\sigma}_{I}(D))\rangle$, and the result follows.
\end{proof}

\begin{example}
Let $R=\CC[x_{0},x_{1},x_{2}]$ be the Cox ring of $\PP^{2}$, $B=(x_0,x_1,x_2)$ and let $I=(x_{0}^{3}x_{1},x_{0}x_{1}x_{2}^{2},x_{2}^{3},x_{1}^{3})$ be a monomial ideal. Computing its Klyachko diagram we obtain $s_{0}=s_{1}=s_{2}=0$, $\Delta^{0}=\{(0,0),(0,1),(0,2)\}$ and $\Delta^{1}=\Delta^{2}=\emptyset$. From Proposition \ref{Corollary:local HH1} we obtain:
\[
\begin{array}{ll}
H^{1}_{B}(I)_{0}=0
	&H^{1}_{B}(I)_{3}=\CC\langle x_{0}^{2}x_{1},x_{0}x_{1}^{2},x_{0}x_{1}x_{2},x_{1}^{2}
			x_{2},x_{1}x_{2}^{2}\rangle\\[1mm]
H^{1}_{B}(I)_{1}=\CC\langle x_{1}\rangle
	&H^{1}_{B}(I)_{4}=\CC\langle x_{0}^{2}x_{1}^{2},x_{0}^{2}x_{1}x_{2},x_{0}x_{1}^{2}
		x_{2},x_{1}^{2}x_{2}^{2}\rangle\\[1mm]
H^{1}_{B}(I)_{2}=\CC\langle x_{0}x_{1},x_{1}^{2},x_{1}x_{2}\rangle
	&H^{1}_{B}(I)_{5}=\CC\langle x_{0}^{2}x_{1}^{2}x_{2}\rangle\\[1mm]
\multicolumn{2}{c}{H^{1}_{B}(I)_{j}=0\quad \text{for}\quad j\geq 6.\hspace*{5mm}}
\end{array}
\]
\end{example}

\section{Application: Hilbert function of monomial ideals}
In this section, we show how to compute the Hilbert function and Hilbert polynomial of a $B-$saturated monomial ideal from its Klyachko diagram. As a consequence we develop a formula for the Hilbert polynomial in terms of the Klyachko diagram.

For any ray $\rho\in\Sigma(1)$, recall that $\cC_{0}^{\rho}:=\{m\in M| \langle m,\rho\rangle\geq0\}$,  $\cC^{\sigma}_{0}=\bigcap_{\rho\in\sigma(1)}\cC^{\rho}$ for $\sigma\in\Sigma$. Recall by Lemma \ref{Lemma:KlyFilt for monomial} that $\{\cC^{\sigma}_{0},\emptyset)\}_{\sigma\in\Sigma}$ is the Klyachko diagram of the monomial ideal $(1)$, or equivalently of $R$. For any multidegree $D\in\ZZ^{r}$, we define $\cC_{0}(D):=\bigcap_{\sigma\in\Sigma_{\max}}\cC_{0}^{\sigma}(D)$, such that $\cC_{0}(D)$ gives a monomial basis of $R_{[D]}$ (see Proposition \ref{Lemma:generators of I}). The following result tells us how to compute the value of the Hilbert function of $I$ from this description.

\begin{proposition}\label{Prop:Hilbert function and Klyachko diagram}
Let $I$ be a $B-$saturated monomial ideal with Klyachko diagram $\{(\cC^{\sigma}_{I},\Delta^{\sigma}_{I})\}_{\sigma\in\Sigma}$. Then, the Hilbert function of $I$ is given by
\[
h_{R/I}(\alpha)=
\left|
\bigcup_{\sigma\in\Sigma_{\max}}\left((\Delta^{\sigma}_{I}(\overline{\alpha})\cap\cC_{0}(\overline{\alpha}))\cup(\cC_{0}(\overline{\alpha})\setminus \cC^{\sigma}_{I}(\overline{\alpha}))\right)
\right|
\]
for any $\alpha\in\Cl(X)$.
\end{proposition}
\begin{proof}
By Lemma \ref{Lemma:generators of I}, there is a bijection between a monomial basis of $I_{\alpha}$ (respectively of $R_{\alpha}$) and $\bigcap_{\sigma\in\Sigma_{\max}}(\cC^{\sigma}_{I}(\overline{\alpha})\setminus\Delta^{\sigma}_{I}(\overline{\alpha}))$ (respectively  $\cC_{0}(\overline{\alpha})$). Thus, 
\[\begin{array}{ll>{\displaystyle}l}
h_{R/I}(\alpha) & = & \left|\cC_{0}(\overline{\alpha})\setminus \left(
\bigcap_{\sigma\in\Sigma_{\max}}(\cC^{\sigma}_{I}(\overline{\alpha})\setminus\Delta^{\sigma}_{I}(\overline{\alpha}))
\right)\right|\\[0.5cm]
& = & \left|
\bigcup_{\sigma\in\Sigma_{\max}}\cC_0(\overline{\alpha})\setminus (\cC^{\sigma}_{I}(\overline{\alpha})\setminus \Delta^{\sigma}_{I}(\overline{\alpha}))
\right|\\[0.5cm]
& = &  \left|
\bigcup_{\sigma\in\Sigma_{\max}}\left((\Delta^{\sigma}_{I}(\overline{\alpha})\cap\cC_{0}(\overline{\alpha}))\cup(\cC_{0}(\overline{\alpha})\setminus \cC^{\sigma}_{I}(\overline{\alpha}))\right)
\right|.
\end{array}
\]
\end{proof}

In the following we remark how the above formula simplifies when $I$ is an ideal generated by monomials with no common factor.

\begin{remark}\label{Remark:sI = 0} Let $I \subset R$ be a monomial ideal and $x^{\uk} \in R$ a monomial. We recall that the Hilbert function of $J = x^{\uk}I$ is 
\begin{equation}\label{Eq:Hilbert function identity}
h_{R/J}(\alpha) = h_{R}(\alpha)-h_{R}(\alpha-[\uk]) + h_{R/I}(\alpha-[\uk]).
\end{equation}
Thus, we can assume that the Klyachko diagram of $I$ has $s_{\rho} = 0$ for any $\rho \in \Sigma(1)$ and, its Hilbert function is 
\begin{equation}\label{Eq:Hilbert function sI=0}
h_{R/I}(\alpha)=
\left|
\bigcup_{\sigma\in\Sigma_{\max}}\left(\Delta^{\sigma}(\overline{\alpha})\cap\cC_{0}(\overline{\alpha})\right)
\right|.
\end{equation}
Otherwise, $I = \left(\prod_{\rho \in \Sigma(1)} x_{\rho}^{s_{\rho}}\right) I_0$ where $I_0$ is a monomial ideal with $s_{\rho} = 0$ for any $\rho \in \Sigma(1)$, and we can compute the Hilbert function of $I$ using (\ref{Eq:Hilbert function identity}). 
\end{remark}

\begin{example}
Let $R=\CC[x_0,x_1,x_2]$ be the Cox ring of $\PP^{2}$ and $I=(x_2^2,$ $x_0x_2,x_0x_1)$ as in Example \ref{Example:Klyachko diagram algorithm} (i). For any $a\in\ZZ$, we set $\overline{a}=(a,0,0)$ and
\[
\Delta^{\sigma_{0}}_{I}(\overline{a})=\{(0,0)\}\quad\Delta^{\sigma_{1}}_{I}(\overline{a})=\{(a,0),(a-1,1)\}\quad
\Delta^{\sigma_{2}}_{I}(\overline{a})=\emptyset.
\]
Since $s_0=s_1=s_2=0$, by (\ref{Eq:Hilbert function sI=0}) we have the following Hilbert function:
\[
h_{R/I}(t)=
\left\{
\begin{array}{@{}ll@{}}
0,&t\leq-1\\
1,&t=0\\
3,&t\geq1.
\end{array}
\right.
\]
In particular, the Hilbert polynomial of $R/I$ is $P_{R/I}\equiv 3$ constant.
\end{example}

In the following result we characterize the Klyachko diagram of a monomial ideal $I$ with constant Hilbert polynomial. In particular, notice that $I$ is necessarily generated by monomials without common factors.

\begin{corollary}\label{Corollary:HilbPoly}
 Let $I$ be a monomial ideal with Klyachko diagram $\{(\cC^{\sigma}_{I},\Delta^{\sigma}_{I})\}_{\sigma \in \Sigma}$. Then, the Hilbert polynomial $P_{R/I}$ of $I$ is constant if and only if $s_{\rho} = 0$ for any $\rho \in \Sigma(1)$ and $\Delta^{\sigma}_{I}$ is finite for any $\sigma \in \Sigma_{\max}$. Moreover, 
\[P_{R/I}(\alpha) = \sum_{\sigma \in \Sigma_{\max}} |\Delta^{\sigma}_{I}|.\]
\end{corollary}
\begin{proof}
The left implication follows directly from Proposition \ref{Prop:Hilbert function and Klyachko diagram} and Remark \ref{Remark:sI = 0}. Conversely, if $s_{\rho}>0$ for some $\rho\in\Sigma(1)$, then there is $\sigma\in\Sigma_{\max}$ such that $\rho\in\sigma(1)$ and $\cC_{0}(\overline{\alpha})\setminus \cC^{\sigma}_{I}(\overline{\alpha})$ increases with $\alpha$, and $P_{R/I}$ would not be constant. Now, assume that there is some $\sigma \in \Sigma_{\max}$ such that $\Delta^{\sigma}_{I}$ is not finite. By construction, $\Delta^{\sigma}_{I}$ contains a set $\Delta'$ of the form 
\[
\{m \in M \mid \langle m, \rho_{i_{1}} \rangle\!=\!k_{1}, \hdots, \langle m, \rho_{i_{l}} \rangle\!=\!k_{l}, \langle m,\rho_{i_{l+1}} \rangle\!\geq\!k_{l+1}, \hdots, \langle m,\rho_{i_{c}} \rangle\!\geq\!k_{c}\}.
\]
Since $\Delta'(\overline{\alpha})$ is not bounded in $\cC_{0}^{\sigma}(\overline{\alpha})$, the number of points in $\Delta'(\overline{\alpha}) \cap \cC_{0}(\overline{\alpha})$ increases with $\alpha$ for $\alpha \gg 0$. Therefore, it follows from (\ref{Eq:Hilbert function sI=0}) that the Hilbert function $h_{R/I}(\alpha)$ of $I$ increases with $\alpha$ for $\alpha \gg 0$. 
\end{proof}

\begin{remark}
Notice that by Corollary \ref{Corollary:HilbPoly} we have characterized all monomial ideals $I\subset R$ with $\dim R/I =1$, in terms of the Klyachko diagram.
\end{remark}

We finish by illustrating Corollary \ref{Corollary:HilbPoly} with the following example.
\begin{example}
Let $R\!=\!\CC[x_0,x_1,x_2,x_3]$ be the Cox ring of $\PP^3$ and $I\!\!=\!\!(x_0x_1,x_2^2,x_1x_2x_3^2)$ as in Example \ref{Example:Klyachko diagram algorithm} (ii). For any $a\in\ZZ$ we set $\overline{a}$ and we have,
\[
\begin{array}{l@{}}
\Delta^{\sigma_{0}}_{I}(\overline{a})=\{(0,d_{2},d_{3})\mid0\leq d_{2},d_{3}\leq 1\}\cup\{(0,d_{2},d_{3})\mid d_{3}\geq 2,0\leq d_{2}\leq 1\}\\
\Delta^{\sigma_{1}}_{I}(\overline{a})=\{(a-d_{2}-d_{3},d_{2},d_{3})\mid0\leq d_{2},d_{3}\leq 1\}\cup
\{(a-d_{3},0,d_{3})\mid d_{3}\geq2\}\\
\Delta^{\sigma_{2}}_{I}(\overline{a})=\emptyset\\
\Delta^{\sigma_{3}}_{I}(\overline{a})=\{(0,d_{2},d_{3})\mid 0\leq d_{2}\leq 1, d_{3}\leq a-d_{2}\}\cup
\{(d_{1},0,a-d_{1})\mid d_{1}\geq 1\}
\end{array}
\]
Counting the number of different points in $\bigcup_{i=0}^{3}\Delta^{\sigma_{i}}_{I}(\overline{a})$, we obtain that $h_{R/I}(0)=1$, $h_{R/I}(1)=4$, $h_{R/I}(2)=8$, and for $a\geq3$, $h_{R/I}(a)=3(a+1)$, for $a\geq 3$. Thus, the Hilbert polynomial of $I$ is $P_{R/I}(a)=3(a+1)$.
\end{example}

\section{Declarations}

{\bf Funding or conflicts of interests.}
The authors have no competing interests to declare that are relevant to the content of this article.

\vspace{5mm}
{\bf Data availability.} Data sharing not applicable to this article as no datasets were generated or analysed during the current study.


\begin{thebibliography}{ll}

\bibitem{Bayer-Sturmfels} D. Bayer and B. Sturmfels. {\em Cellular resolutions of monomial modules}. J. Reine
Angew. Math. {\bf 502} (1998), 123--140.

\bibitem{Batyrev-Cox} V. V. Batyrev and D. A. Cox {\em On the Hodge structure of projective hypersurfaces in toric varieties} Duke Math. J. {\bf 75:2} (1994), 293--338.

\bibitem{Cox1} D. A. Cox, {\em The Homogeneous Coordinate Ring of a Toric Variety}. J. Algebr. Geom. {\bf 4:1} (1995), 17--50.

\bibitem{CLS} D. A. Cox, J. B. Little and H. K. Schenck {\em Toric varieties} GSM {\bf 124} American Mathematical Society, 2011.

\bibitem{Das-Dey-Khan} J. Dasgupta, A. Dey, B. Khan, {\em Stability of Equivariant Vector Bundles over Toric Varieties} Doc. Math. {\bf 25} (2020), 1787--1833.

\bibitem{Di-Rocco-Jabbusch-Smith} S. Di Rocco, K. Jabbusch and G. G. Smith, {\em Toric vector bundles and parliaments of polytopes} Trans. Amer. Math. Soc. {\bf 370} (2018), 7715--774.

\bibitem{Eagon-Reiner} J. A.  Eagon  and  V.  Reiner.   {\em Resolutions  of  Stanley-Reisner  rings  and  Alexander duality}. J. Pure Appl. Algebra. {\bf 130:3} (1998), 265--275

\bibitem{EisenbudBook} D. Eisenbud, {\em Commutative algebra with a view toward algebraic geometry}. GTM {\bf 150}, Springer, 2004.

\bibitem{Eis-Mus-Sti} D. Eisenbud, M. Musta\c{t}\u{a} and M. Stillman {\em Cohomology on Toric Varieties and Local Cohomology with Monomial Supports}J. Symbolic Comput. {\bf 29:4-5} (2000), 583--600.

\bibitem{Eliahou-Kervaire} S. Eliahou and M. Kervaire {\em Minimal resolutions of some monomial ideals}. J.
Algebra. {\bf 129:1} (1990), 1--25.

\bibitem{Faridi} S. Faridi, {\em Facet ideal of a simplicial complex}. manuscr. math. {\bf 109:2} (2002), 159--174.

\bibitem{Hei-Rat-Shah} W. Heinzer, L. J. Ratliff Jr., K. Shah
{\em Parametric decompositions of monomial ideals (I)}. Houston J. Math. {\bf 21:1} (1995), 29--52 

\bibitem{Herzog-Hibi} J. Herzog and T. Hibi, {\em Monomial ideals}. GTM {\bf 260} Springer, 2011.

\bibitem{M2} D. R. Grayson and M. E. Stillman, Macaulay2, a software system for research in Algebraic Geometry. Available
at http://www.math.uiuc.edu/Macaulay2/.

\bibitem{Her-Mus-Pay} M. Hering, M. Musta\c{t}\u{a} and S. Payne, {\em Positivity properties of toric vector bundles}. Annales de l'Institut Fourier, {\bf 60:2} (2010), 607--640. 

\bibitem{Khan-Das} B. Khan and J. Dasgupta, {\em Toric vector bundles on Bott tower}. Bull. des Sci. Math. {\bf 155} (2019), 74--91.

\bibitem{Kly89} A. A. Klyachko, {\em Equivariant bundles on toral varieties.} Izv. Akad. Nauk SSSR Ser. Mat. {\bf 35:2} (1990), 337--375.

\bibitem{Kly91} A. A. Klyachko, {\em Vector Bundles and Torsion Free Sheaves on the Projective Plane}. Preprint Max Planck Institut für Mathematik, (1991).

\bibitem{Macaulay} F. S. Macaulay, {\em Some properties of enumeration in the theory of modular systems}. Proc. London Math. Soc. {\bf 26} (1927), 531--555.

\bibitem{MR-S} R. M. Mir\'o-Roig and M. Salat, {\em Multigraded Castelnuovo-Mumford regularity via Klyachko filtrations}. Forum Mathematicum, {\bf 34:1} (2022), pp. 41-60.

\bibitem{Pay} S. Payne, {\em Moduli of toric vector bundles} Compositio Math. {\bf 144:5} (2008) 1199--1213

\bibitem{Perling} M. Perling, {\em Resolutions and Moduli for Equivariant Sheaves over Toric Varieties}, PhD thesis, Universit\"at Kaiserslautern, 2003

\bibitem{Taylor} D. Taylor. {\em Ideals Generated By Monomials in an R-sequence}. PhD thesis, University of Chicago, 1961.
\end{thebibliography}
\end{document}